\documentclass[12pt,a4paper]{article}
\usepackage{amssymb,amsmath,amsthm}
\usepackage{latexsym}

\newtheorem{theorem}{Theorem}
\newtheorem{lemma}{Lemma}
\newtheorem{proposition}{Proposition}
\newtheorem{corollary}{Corollary}
\newtheorem{remark}{Remark}
\newtheorem{definition}{Definition}
\numberwithin{equation}{section}
\numberwithin{theorem}{section}
\numberwithin{lemma}{section}
\numberwithin{proposition}{section}
\numberwithin{corollary}{section}
\numberwithin{remark}{section}
\numberwithin{definition}{section}

\begin{document}
\title{Navier-Stokes flow past a rigid body: \\
attainability of steady solutions as limits \\
of unsteady weak solutions, \\
starting and landing cases}
\author{Toshiaki Hishida \\
Graduate School of Mathematics, Nagoya University \\
Nagoya 464-8602, Japan \\
\texttt{hishida@math.nagoya-u.ac.jp} \\
and \\
Paolo Maremonti \\
Dipartimento di Matematica e Fisica  \\
Universit\`a degli Studi della Campania Luigi Vanvitelli \\
I-81100 Caserta, Italy \\
\texttt{paolo.maremonti@unicampania.it}}
\date{}
\maketitle
\begin{abstract}
Consider the Navier-Stokes flow in 3-dimensional exterior domains,
where a rigid body is translating with prescribed translational velocity
$-h(t)u_\infty$ with constant vector $u_\infty\in\mathbb R^3\setminus\{0\}$.
Finn raised the question whether his steady solutions are attainable as limits
for $t\to\infty$ of unsteady solutions starting from motionless state when
$h(t)=1$ after some finite time and $h(0)=0$ (starting problem).
This was affirmatively solved by Galdi, Heywood and Shibata \cite{GHS}
for small $u_\infty$.
We study some generalized situation in which unsteady solutions start from
large motions being in $L^3$.
We then conclude that the steady solutions for small $u_\infty$ are still 
attainable as limits of evolution of those fluid motions which are found as 
a sort of weak solutions.
The opposite situation, in which $h(t)=0$ after some finite time and $h(0)=1$
(landing problem), is also discussed.
In this latter case, the
rest state
is attainable no matter how large $u_\infty$ is.


\end{abstract}

\section{Introduction and results}\label{intro}
Let us consider a viscous incompressible flow past an obstacle in 3D,
which is a translating rigid body with a prescribed velocity $-hu_\infty$,
where $u_\infty\in\mathbb R^3\setminus\{0\}$ is a constant vector
and the function $h=h(t)$ describes the transition of the translational
velocity of the body.
In the frame attached to the body, the motion of the fluid obeys
the exterior problem for the Navier-Stokes system
\begin{equation}
\begin{split}
&\partial_tu+u\cdot\nabla u=\Delta u-\nabla p_u-hu_\infty\cdot\nabla u, \\
&\mbox{div $u$}=0, \\
&u|_{\partial\Omega}=-hu_\infty, \\
&u\to 0 \quad\mbox{as $|x|\to\infty$}, \\
\end{split}
\label{NS-1}
\end{equation}
where $\Omega$ denotes the exterior of the body in $\mathbb R^3$
with smooth boundary $\partial\Omega$.
The unknown functions are the velocity field
$u=(u_1(x,t),u_2(x,t),u_3(x,t))$ and the associated pressure $p_u=p_u(x,t)$.

Suppose both the fluid and the body are initially at rest, that is,
$u(\cdot,0)=0$ and $h(0)=0$.
If the body starts to move from the rest state until the terminal velocity
$-u_\infty$ at an instant $T_0>0$ and, afterwards,
$h(t)=1$ for $t\geq T_0$,
then the large time behavior of the solution $u(x,t)$ to \eqref{NS-1}
subject to the initial condition $u(\cdot,0)=0$ would be related to
the steady problem
\begin{equation}
\begin{split}
&u_s\cdot\nabla u_s=\Delta u_s-\nabla p_{u_s}-u_\infty\cdot\nabla u_s, \\
&\mbox{div $u_s$}=0, \\
&u_s|_{\partial\Omega}=-u_\infty, \\
&u_s\to 0 \quad\mbox{as $|x|\to\infty$}. \\
\end{split}
\label{NS-s}
\end{equation}
Indeed, in this situation, 
Finn \cite{F1} raised the question whether $u(x,t)$ converges to
$u_s(x)$ as $t\to\infty$ in a sense as long as 
$u_\infty\in\mathbb R^3\setminus\{0\}$ is small enough (Finn's starting problem).
If that is the case, the steady flow $u_s(x)$ is said to be
``attainable'' by following the terminology of Heywood \cite{He1},
who gave a partial answer to the starting problem.
Note that the steady problem \eqref{NS-s} with sufficiently small
$u_\infty\in\mathbb R^3\setminus\{0\}$ possesses a unique solution
$u_s$, what is called the physically reasonable solution,
due to Finn \cite{F2} himself.
On account of its anisotropic behavior with wake property,
the solution $u_s(x)$ enjoys better summability
$u_s\in L^q(\Omega)$ for every $q>2$
(than the case where the body is at rest), see \eqref{steady} below,
however, still infinite energy
$u_s\notin L^2(\Omega)$ because the net force exerted by the fluid
to $\partial\Omega$ cannot vanish when the external force is absent, see 
Finn \cite{F60} and Galdi \cite{Ga-b}.
It is reasonable to look for a solution $u(x,t)$ of the form
$u(x,t)=h(t)u_s(x)+v(x,t)$
and to expect $u(t)\in L^2(\Omega)$ since $u(0)=0$,
however, in this case,
$v(t)\notin L^2(\Omega)$ follows from $u_s\notin L^2(\Omega)$
and thus the energy method is not enough to construct the perturbation $v(t)$.
Thus the problem had remained open until Kobayashi and Shibata \cite{KS}
developed the $L^q$-$L^r$ decay estimate of the Oseen semigroup,
see \eqref{d-1}--\eqref{d-2} below.
Finally, by making use of this estimate, the starting problem
from the rest state
was completely solved by Galdi, Heywood and Shibata \cite{GHS}.

In the present paper we intend to provide further contributions to this issue
for its better understanding.
It would be worth while studying more possibilities of
attainablity of the steady flow $u_s$.
The aim is to find out many solutions to \eqref{NS-1},
which converge to $u_s$ as $t\to\infty$,
even if starting from large motions of both the fluid and the body,
that is, the initial velocity
\begin{equation}
u(x,0)=u_0(x)
\label{IC}
\end{equation}
can be large with infinite energy and $h(0)$ is large, too.
We take $u_0$ from $L^3(\Omega)$, as usual, or even from
$L^{3,\infty}_0(\Omega)$, the completion of $C_0^\infty(\Omega)$
in the Lorentz space (weak-$L^3$ space) $L^{3,\infty}(\Omega)$,
together with the compatibility conditions
\begin{equation}
\mbox{div $u_0$}=0, \qquad
\nu\cdot(u_0+h(0)u_\infty)|_{\partial\Omega}=0,
\label{compati}
\end{equation}
where $\nu$ stands for the outer unit normal to $\partial\Omega$ and
the latter condition is understood in the sense of normal trace.
The function $h=h(t)$ is assumed to satisfy
\begin{equation}
h\in C^{1,\theta}([0,\infty))\quad\mbox{for some $\theta\in (0,1)$},
\label{h-1}
\end{equation}
\begin{equation}
h(t)=1 \quad\mbox{on $[T_0,\infty)$ for some $T_0>0$}.
\label{h-2}
\end{equation}
The main result on the starting problem reads as follows.
\begin{theorem}
There exists a 
constant
$\delta >0$ with the following property:
If $u_\infty\in\mathbb R^3\setminus\{0\}$ fulfills $|u_\infty|\leq\delta$,
then, for every $u_0\in L^{3,\infty}_0(\Omega)$ with \eqref{compati}
and for every function $h(t)$ satisfying \eqref{h-1}--\eqref{h-2},
problem \eqref{NS-1} subject to \eqref{IC} admits at least one solution
$u(x,t)$ which enjoys
\begin{equation}
\|u(t)-u_s\|_{L^\infty(\Omega)}=O(t^{-1/2})
\label{attain-1}
\end{equation}
as $t\to\infty$, where $u_s$ is a unique solution to \eqref{NS-s}.
\label{starting}
\end{theorem}

We stress that the small constant $\delta$ 
in Theorem \ref{starting} is independent of $u_0$ and $h$.
Our global solution is a sort of weak solution, to be precise, it is of the form
\begin{equation}
u(x,t)=h(t)u_s+\widetilde U(x,t)+w(x,t),
\label{decompo}
\end{equation}
where $\widetilde U(x,t)$ is an auxiliary function (regular enough for $t>0$),
while $w(x,t)$ is the so-called Leray-Hopf weak solution \cite{L2}, 
\cite{Ho}, \cite{Mas}.
The idea to solve the Navier-Stokes initial value problem with
large initial data in $L^3$ (or $L^{3,\infty}_0$) is due to Maremonti \cite{M},
in which a solution to \eqref{NS-1} with $u_\infty=0$ subject to \eqref{IC}
is constructed in the form
$u(t)=e^{-tA}u_0+w(t)$ with a Leray-Hopf weak solution $w(t)$,
where $e^{-tA}$ denotes the Stokes semigroup.
The similar approach was adopted also by \cite{BaS}, \cite{SS}.
In the case under consideration of this paper, the pair
\[
v(x,t):=u(x,t)-h(t)u_s(x), \qquad
p_v(x,t):=p_u(x,t)-h(t)p_{u_s}(x)
\]
should obey
\begin{equation}
\begin{split}
&\partial_tv+v\cdot\nabla v+h(u_s\cdot\nabla v+v\cdot\nabla u_s)
=\Delta v-\nabla p_v-hu_\infty\cdot\nabla v+g, \\
&\mbox{div $v$}=0, \\
&v|_{\partial\Omega}=0, \\
&v\to 0 \quad\mbox{as $|x|\to\infty$}, \\
&v(\cdot,0)=v_0:=u_0-h(0)u_s \\
\end{split}
\label{NS-2}
\end{equation}
with the forcing term
\begin{equation}
g(x,t):=-h^\prime u_s+(h-h^2)(u_s+u_\infty)\cdot\nabla u_s,
\label{force-1}
\end{equation}
where $h^\prime=\frac{dh}{dt}$.
There would be several possibilities of choice of the auxiliary function
$\widetilde U(x,t)$ in \eqref{decompo},
which plays the same role as $e^{-tA}u_0$ in \cite{M}.
With any choice of $\widetilde U(x,t)$ at hand, we subtract this 
function from $v(x,t)$ to see that the remaining part 
$w(x,t):=v(x,t)-\widetilde U(x,t)$ together with the associated pressure
$p_w$ satisfies
\begin{equation}
\begin{split}
&\partial_tw+w\cdot\nabla w
+\widetilde U\cdot\nabla w+w\cdot\nabla\widetilde U
+h(u_s\cdot\nabla w+w\cdot\nabla u_s) \\
&\qquad\qquad =\Delta w-\nabla p_w-hu_\infty\cdot\nabla w+f, \\
&\mbox{div $w$}=0, \\
&w|_{\partial\Omega}=0, \\
&w\to 0 \quad\mbox{as $|x|\to\infty$}, \\
&w(\cdot,0)=w_0:=v_0-\widetilde U(\cdot,0),
\end{split}
\label{NS-3}
\end{equation}
for some vector field $f=f(x,t)$ as the new forcing term whenever
\[
\mbox{div $\widetilde U$}=0, \qquad
\widetilde U|_{\partial\Omega}=0, \qquad
\widetilde U\to 0\; (|x|\to\infty).
\]
Besides these conditions, the auxiliary function
$\widetilde U(x,t)$ must be taken so that
$f\in L^2_{loc}([0,\infty); H^{-1}(\Omega))$ as well as
$w_0\in L^2(\Omega)$ in order to look for $w(x,t)$ as the Leray-Hopf weak
solution with the strong energy inequality
\begin{equation}
\begin{split}
&\quad \frac{1}{2}\|w(t)\|^2_{L^2(\Omega)}
+\int_s^t\|\nabla w\|^2_{L^2(\Omega)} d\tau \\
&\leq\frac{1}{2}\|w(s)\|^2_{L^2(\Omega)}
+\int_s^t\langle (hu_s+\widetilde U)\otimes w, \nabla w \rangle d\tau
+\int_s^t\langle f, w\rangle d\tau
\end{split}
\label{SEI}
\end{equation}
for $s=0$, a.e. $s>0$ and all $t\geq s$.
As the auxiliary function, in this paper,
we will take the solution of the non-autonomous Oseen
initial value problem in the whole space $\mathbb R^3$ together with
a correction term,
see \eqref{auxiliary} and \eqref{auxi-eq}.
Then the forcing term $f(x,t)$ is given by \eqref{data} together
with \eqref{auxi-force}.

For the proof of attainability \eqref{attain-1} of the steady flow,
a crucial step is to find out a large instant $\bar t >0$ such that
$w(\bar t)$ is small enough in $L^3(\Omega)$.
It is then possible to construct a global strong solution from $\bar t$
with some decay properties,
particularly $L^\infty$-decay like $O(t^{-1/2})$,
which can be identified with the weak solution $w(t)$
by the strong energy inequality \eqref{SEI}.
Indeed this strategy itself is quite classical since the celebrated
paper by Leray \cite{L2},
but there are some details to make 
$\|w(\bar t)\|_{L^3(\Omega)}$ small
at a suitable $\bar t$.
This is by no means obvious
since the RHS of \eqref{SEI} is growing for $t\to\infty$.
One would raise the question whether Theorem \ref{starting} still holds 
for $u_0\in L^{3,\infty}(\Omega)$ 
(that is strictly larger than $L^{3,\infty}_0(\Omega)$). 
For such data, unfortunately, the behavior of the auxiliary function 
$\widetilde U(t)$ near $t=0$ is critical and this prevents us 
from constructing the weak solution $w(t)$.

It is also interesting to consider the opposite situation
(landing problem), in which the body is initially translating with
velocity $-u_\infty$ and it stops at an instant $T_0$ and is kept
afterwards at rest, that is,
\begin{equation}
h(t)=0 \quad\mbox{on $[T_0,\infty)$ for some $T_0>0$}; \qquad h(0)=1.
\label{h-3}
\end{equation}
The following result on the landing problem tells us that the rest
state is attainable no matter how large $u_\infty$ is.
\begin{theorem}
For every $u_\infty\in \mathbb R^3\setminus\{0\}$,
$u_0\in L^{3,\infty}_0(\Omega)$ with \eqref{compati} and
$h(t)$ satisfying \eqref{h-3} as well as \eqref{h-1},
problem \eqref{NS-1} subject to \eqref{IC} admits at least one solution $u(x,t)$
which enjoys
\begin{equation}
\|u(t)\|_{L^\infty(\Omega)}=O(t^{-1/2})
\label{attain-2}
\end{equation}
as $t\to\infty$.
\label{landing}
\end{theorem}
The idea of the proof of Theorem \ref{landing} is the same as the one for
the starting problem.
For every $u_\infty\in \mathbb R^3\setminus\{0\}$ the steady problem
\eqref{NS-s} admits at least one solution $u_s(x)$ with finite Dirichlet
integral $\nabla u_s\in L^2(\Omega)$ (the Leray class), see Leray \cite{L1}.
It also follows from the result of Babenko \cite{Ba}, Galdi 
\cite{Ga92}, \cite{Ga-b},
Farwig and Sohr \cite{FaS98}
that any solution of the Leray class
eventually becomes the physically reasonable solution in the sense 
of Finn \cite{F1}, \cite{F2}.
Since we would have several solutions unless $u_\infty$ is small,
we fix a steady flow $u_s(x)$ arbitrarily
among them and look for the solution $u(x,t)$ to \eqref{NS-1}
of the form \eqref{decompo}.
It would be interesting to ask
sharper $L^\infty$-decay like $o(t^{-1/2})$ in \eqref{attain-2} as well as
\eqref{attain-1};
in fact, this is possible for \eqref{NS-1} with $u_\infty=0$ subject to
\eqref{IC} when $u_0\in L^{3,\infty}_0$ is small enough, see \cite{M0}.
On account of the presence of the forcing term
(especially $\widetilde U\cdot\nabla\widetilde U$, see \eqref{data}),
it does not seem to be clear whether
$\|w(t)\|_{L^\infty(\Omega)}=o(t^{-1/2})$, however,
one could take another way in which one constructs directly
a strong solution $v(t)$ on $[\bar t,\infty)$
with a suitable $\bar t$ for \eqref{NS-2},
instead of $w(t)$, such that
$\|v(t)\|_{L^\infty(\Omega)}=o(t^{-1/2})$ as $t\to\infty$.

This paper concerns the attainability,
while the stability of the steady flow was extensively studied,
see for instance \cite{Shi}, \cite{ES2}, \cite{Ko}
and the references therein.
The paper is organized as follows.
After some preliminaries in the next section,
we choose the auxiliary function $\widetilde U(x,t)$
in \eqref{decompo} and derive several properties in section \ref{auxi}.
In section \ref{weak} we construct a weak solution $w(t)$ to the initial value
problem \eqref{NS-3} and deduce the strong energy inequality \eqref{SEI}.
In section \ref{strong-iden} we make use of the $L^q$-$L^r$ 
decay estimate of the Oseen
semigroup (\cite{KS}) to construct a strong solution to \eqref{NS-3}
on $[\bar t,\infty)$ whenever $w(\bar t)$ is small in $L^3(\Omega)$.
We further show that this solution 
is identified with the weak solution on $[\bar t,\infty)$.
The final section is devoted to finding $\bar t >0$,
at which 
$\|w(\bar t)\|_{L^3(\Omega)}$
is actullay small enough,
to accomplish the proof of Theorems \ref{starting} and \ref{landing}.

\section{Preliminaries}\label{pre}

We start with introducing notation.
Given a domain $D\subset\mathbb R^3$, $1\leq q\leq\infty$, and integer
$k\geq 0$, we denote by $L^q(D)$ and 
by
$W^{k,q}(D)$
the standard Lebesgue and Sobolev spaces, respectively.
We simply write the norm
$\|\cdot\|_{q,D}=\|\cdot\|_{L^q(D)}$ and even
$\|\cdot\|_q=\|\cdot\|_{q,\Omega}$,
where $\Omega$ is the exterior domain under consideration.
Let $C_0^\infty(D)$ be the class of smooth functions with compact support in $D$.
We denote by $W^{k,q}_0(D)$ the completion of $C_0^\infty(D)$ in
$W^{k,q}(D)$, and by
$W^{-1,q}(D)$ the dual space of $W^{1,q^\prime}_0(D)$, where
$1/q^\prime+1/q=1$ and $q\in (1,\infty)$.
By $\langle\cdot,\cdot\rangle$ we denote various duality pairings
on $\Omega$.
When $q=2$, we write
$H^k(D)=W^{k,2}(D)$, $H_0^1(D)=W^{1,2}_0(D)$ and
$H^{-1}(D)=W^{-1,2}(D)$, respectively.

Let us introduce the Lorentz spaces
(for details, see Bergh and L\"ofstr\"om \cite{BL}).
Given a measurable function $f$ on a domain $D$, we set
\begin{equation*}
\begin{split}
&m_f(\tau)=\left|\{x\in D;\, |f(x)|>\tau\}\right|, \qquad \tau >0, \\
&f^*(t)=\inf\{\tau>0;\, m_f(\tau)\leq t\}, \qquad t>0,
\end{split}
\end{equation*}
where $|\cdot|$ stands for the Lebesgue measure.
Let $1<q<\infty$ and $1\leq r\leq \infty$, then the space 
$L^{q,r}(D)$ consists of all measurable functions $f$ on $D$
which satisfy
\begin{equation}
\begin{split}
&\left(\int_0^\infty\left\{t^{1/q}f^*(t)\right\}^r\frac{dt}{t}
\right)^{1/r}<\infty \qquad (1\leq r<\infty), \\
&\sup_{t>0}\,t^{1/q}\,f^*(t)<\infty \qquad (r=\infty).
\end{split}
\end{equation}
Each of those quantities is a quasi-norm, however,
it is possible to introduce an equivalent norm
$\|\cdot\|_{q,r,D}$ by use of the average function.
Then $L^{q,r}(D)$ endowed with $\|\cdot\|_{q,r,D}$ is a Banach space,
called the Lorentz space.
We simply write
$\|\cdot\|_{q,r}=\|\cdot\|_{q,r,\Omega}$.
Note that
$L^{q,q}(D)=L^q(D)$ and that
$L^{q,r_0}(D)\subset L^{q,r_1}(D)$ if $r_0\leq r_1$.
The space $L^{q,\infty}(D)$ is well known as the weak-$L^q$ space,
in which $C_0^\infty(D)$ is not dense.
Let us define the space
$L^{q,\infty}_0(D)$ by the completion of $C_0^\infty(D)$ in 
$L^{q,\infty}(D)$.
The Lorentz space can be also constructed via real interpolation
\[
L^{q,r}(D)=\big(L^1(D), L^\infty(D)\big)_{1-1/q,r}
\]
from which the reiteration theorem in the interpolation
theory
leads to
\[
L^{q,r}(D)=
\big(L^{q_0,r_0}(D), L^{q_1,r_1}(D)\big)_{\theta,r}
\]
together with
\begin{equation}
\|f\|_{q,r,D}\leq C\|f\|_{q_0,r_0,D}^{1-\theta}\|f\|_{q_1,r_1,D}^\theta
\label{interpo-1}
\end{equation}
for all $f\in L^{q_0,r_0}(D)\cap L^{q_1,r_1}(D)\subset L^{q,r}(D)$
provided that
\[
1<q_0<q<q_1<\infty, \quad
\frac{1}{q}=\frac{1-\theta}{q_0}+\frac{\theta}{q_1}, \quad
1\leq r_0, r_1, r\leq\infty.
\]
We have the Lorentz-H\"older and Lorentz-Sobolev inequalities,
but the only cases we need in this paper are
\begin{equation}
\|fg\|_{r,s,D}\leq \|f\|_{3,\infty,D}\|g\|_{q,s,D}, \qquad
\frac{1}{r}=\frac{1}{3}+\frac{1}{q},\quad  q,\,r\in (1,\infty),
\label{LH}
\end{equation}
\begin{equation}
\|g\|_{q_*,s}\leq C\|\nabla g\|_{q,s}, \qquad
\frac{1}{q_*}=\frac{1}{q}-\frac{1}{3},\quad q\in (1,3),
\label{LS}
\end{equation}
where $1\leq s\leq \infty$.
In what follows the same symbols for vector and scalar function spaces
are adopted as long as there is no confusion.

Let us introduce the solenoidal function spaces over the exterior 
domain $\Omega$.
The space $C^\infty_{0,\sigma}(\Omega)$ consists of all divergence
free vector fields whose components are in
$C_0^\infty(\Omega)$.
Let $1<q<\infty$.
We denote by $L^q_\sigma(\Omega)$
the completion of $C^\infty_{0,\sigma}(\Omega)$ in $L^q(\Omega)$.
Then it is characterized as
\[
L^q_\sigma(\Omega)=
\{u\in L^q(\Omega);\, \mbox{div $u$}=0,\, \nu\cdot u|_{\partial\Omega}=0\},
\]
where $\nu\cdot u|_{\partial\Omega}$ stands for the normal trace of $u$.
The space $L^q(\Omega)$ of vector fields admits the Helmholtz decomposition
\[
L^q(\Omega)=L^q_\sigma(\Omega)\oplus
\{\nabla p\in L^q(\Omega);\, p\in L^q_{loc}(\overline{\Omega})\}
\]
which was proved by Miyakawa \cite{Mi} and by Simader and Sohr \cite{SiS}.
When $q=2$, it is the orthogobal decomposition.
We have the same result for the whole space $\mathbb R^3$ as well.

By using the projection 
$\mathbb P: L^q(\Omega)\to L^q_\sigma(\Omega)$
associated with the decomposition above,
we define the Stokes operator $A$ by
\[
D_q(A)=W^{2,q}(\Omega)\cap W^{1,q}_0(\Omega)\cap L^q_\sigma(\Omega), \qquad
Af=-\mathbb P\Delta f.
\]
When $q=2$, it is a nonnegative self-adjoint operator in $L^2_\sigma(\Omega)$ and
\[
\langle A^{1/2}f, A^{1/2}g\rangle=\langle\nabla f, \nabla g\rangle, \qquad
\mbox{for $f,\, g\in D_2(A^{1/2})=H_{0,\sigma}^1(\Omega)$},
\]
where the space $H_{0,\sigma}^1(\Omega)$ denotes the completion of
$C_{0,\sigma}^\infty(\Omega)$ in $H^1(\Omega)$.
Due to Solonnikov \cite{Sol}, Giga \cite{Gi} and Farwig and Sohr \cite{FaS},
we know the generation of an analytic semigroup (the Stokes semigroup)
$\{e^{-tA}\}_{t\geq 0}$ on
$L^q_\sigma(\Omega)$.
Furthermore, it is uniformly bounded
$\|e^{-tA}f\|_q\leq C\|f\|_q$
by the result of Borchers and Sohr \cite{BS1}.
Given a constant vector $u_\infty\in\mathbb R^3$,
let us define the Oseen operator $A_{u_\infty}$ by
\[
D_q(A_{u_\infty})=D_q(A), \qquad
A_{u_\infty}f=-\mathbb P[\Delta f-u_\infty\cdot\nabla f].
\]
Then, by a simple perturbation argument, see Miyakawa \cite{Mi},
it is verified that the operator $-A_{u_\infty}$ also generates an
analytic semigroup (the Oseen semigroup)
$\{e^{-tA_{u_\infty}}\}_{t\geq 0}$ on
$L^q_\sigma(\Omega)$.
In \cite{KS} Kobayashi and Shibata (see also Enomoto and Shibata
\cite{ES1}, \cite{ES2}) developed the $L^q$-$L^r$ estimates
\begin{equation}
\|e^{-tA_{u_\infty}}f\|_r
\leq Ct^{-\alpha}\|f\|_q \qquad
(1<q\leq r\leq\infty,\, q\neq\infty),
\label{d-1}
\end{equation}
\begin{equation}
\|\nabla e^{-tA_{u_\infty}}f\|_r
\leq Ct^{-\alpha-1/2}\|f\|_q \qquad
(1<q\leq r\leq 3),
\label{d-2}
\end{equation}
for all $t>0$, where $\alpha=(3/q-3/r)/2$.
They also showed that, for each $K>0$,
the constant 
$C=C(K;\,q,r)>0$
in \eqref{d-1}--\eqref{d-2}
can be taken uniformly with respect to $u_\infty\in\mathbb R^3$ satisfying
$|u_\infty|\leq K$.
Therefore, their result includes the $L^q$-$L^r$ estimates of the Stokes semigroup
as a special case, however, even before, both \eqref{d-1} and \eqref{d-2}
(case $u_\infty=0$)
had been established by Iwashita \cite{I}, Chen \cite{C}
(case $r=\infty$)
and Maremonti and Solonnikov \cite{MSol}.
For later use, let us give a supplement about the Oseen operator, 
which is $m$-accretive in $L^2_\sigma(\Omega)$. 
Since both
$1+A_{u_\infty}$ 
and $1+A$ are invertible, we have
\[
\|Af\|_2\leq C\|(1+A_{u_\infty})f\|_2, \qquad
\|A_{u_\infty}f\|_2\leq C\|(1+A)f\|_2,
\]
for $f\in D_2(A)$. 
Then the Heinz-Kato inequality for $m$-accretive operators implies that
\begin{equation}
\|\nabla f\|_2=\|A^{1/2}f\|_2\leq C\|(1+A_{u_\infty})^{1/2}f\|_2
\label{HK}
\end{equation}
for all $f\in D_2(A_{u_\infty}^{1/2})=D_2(A^{1/2})=H_{0,\sigma}^1(\Omega)$ 
with some constant $C=C(|u_\infty|)>0$.

We next consider the boundary value problem for the equation
of continuity
\[
\mbox{div $w$}=f \;\;\mbox{in $D$}, \qquad
w|_{\partial D}=0,
\]
where $D$ is a bounded domain in $\mathbb R^3$ with Lipschitz boundary
$\partial D$.
Let $1<q<\infty$.
Given $f\in L^q(D)$ with compatibility condition 
$\int_D f=0$,
there are a lot of solutions, some of which were found by many authors,
see Galdi \cite[Notes for Chapter III]{Ga-b}.
Among them a particular solution discovered by Bogovskii \cite{B}
is useful to recover the solenoidal condition in a cut-off procedure
on account of some fine properties of his solution.
The operator $f\mapsto \mbox{his solution $w$}$,
called the Bogovskii operator, is well defined as follows
(for details, see Galdi \cite{Ga-b}, Borchers and Sohr \cite{BS2}):
there is a linear operator
$\mathbb B: C_0^\infty(D)\to C_0^\infty(D)^3$ such that,
for $1<q<\infty$ and $k\geq 0$ integers,
\begin{equation}
\|\nabla^{k+1}\mathbb Bf\|_{q,D}\leq C\|\nabla^kf\|_{q,D}
\label{bog-est-1}
\end{equation}
with some $C=C(D,q,k)>0$ and that
\begin{equation}
\mbox{div $\mathbb Bf$}=f \qquad
\mbox{if}\;\; \int_D f(x)\,dx=0,
\label{bogov}
\end{equation}
where the constant $C$ is invariant with respect to dilation of the domain $D$.
By continuity, $\mathbb B$ is extended uniquely to a bounded operator
from $W^{k,q}_0(D)$ to $W^{k+1,q}_0(D)^3$.
It is obvious by real interpolation that several estimates in the Lorentz
norm similar to \eqref{bog-est-1} are available as well;
for instance, we have
\begin{equation}
\|\nabla \mathbb Bf\|_{q,\infty,D}\leq C\|f\|_{q,\infty,D},
\label{bog-est-L}
\end{equation}
for every $f\in L^{q,\infty}(D)$ and $q\in (1,\infty)$.
By Geissert, Heck and Hieber \cite[Theorem 2.5]{GHH},
$\mathbb B$ can be also extended to a bounded operator 
from $W^{1,q^\prime}(D)^*$ to $L^q(D)^3$, that is,
\begin{equation}
\|\mathbb Bf\|_{q,D}\leq C\|f\|_{W^{1,q^\prime}(D)^*},
\label{bog-est-2}
\end{equation}
where $1/q^\prime+1/q=1$.
Note that this is not true from $W^{-1,q}(D)$ 
to $L^q(D)^3$, see Galdi \cite[Chapter III]{Ga-b}.
Finally, we mention a sort of commutator estimate between $\mathbb B$
and the Laplacian.
Let $f\in W^{2,q}(D)$.
We fix $\eta\in C_0^\infty(D)$ to find
\begin{equation}
\|\Delta\mathbb B[\eta f]-\mathbb B[\Delta(\eta f)]\|_{q,D}
\leq C\|f\|_{q,D}.
\label{bog-est-3}
\end{equation}
Indeed this is rather restricted form, but it is enough for later use,
see Lemma \ref{F-class}.
By the condition above on the domain $D$, see
Galdi \cite[Lemma III.3.4]{Ga-b}, analysis can be reduced to
the case in which $D$ is star-shaped with respect to a ball
$B$, where $\overline{B}\subset D$.
In this case, the solution found by Bogovskii \cite{B} is of
the form (in 3D case)
\[
\mathbb B[\eta f](x)=
\int_D\Gamma_\kappa(x-y,y)(\eta f)(y)dy
\]
with
\[
\Gamma_\kappa(z,y)=z\int_1^\infty
\kappa(y+\tau z)\tau^2 d\tau,
\]
where $\kappa\in C_0^\infty(B)$ is fixed so that
$\int_B\kappa =1$.
Set
\[
{\cal B}_j[\eta f](x)=
\int_D\Gamma_{\partial_j\kappa}(x-y,y)(\eta f)(y)dy \qquad (j=1,2,3).
\]
Then we have
\[
\partial_j\mathbb B[\eta f]-\mathbb B[\partial_j(\eta f)]
={\cal B}_j[\eta f]
\]
for each $j=1,2,3$, and, thereby,
\[
\Delta\mathbb B[\eta f]-\mathbb B[\Delta(\eta f)]
=\sum_j{\cal B}_j[\partial_j(\eta f)]
+\sum_j\partial_j{\cal B}_j[\eta f].
\]
Since the operator ${\cal B}_j$ satisfies the same estimates as in
\eqref{bog-est-1} and \eqref{bog-est-2} in spite of
$\int_B \partial_j\kappa\neq 1$ (which is related only to whether
\eqref{bogov} holds),
the formula above leads to \eqref{bog-est-3}.

\section{Auxiliary function}\label{auxi}

In this section we construct an auxiliary function $\widetilde U(x,t)$ in 
\eqref{decompo}.
We begin with knowledge about the steady problem \eqref{NS-s}.
Due to Finn \cite{F2}, Galdi \cite{Ga-b}, Farwig \cite{Fa}
and Shibata \cite{Shi}, there are constants $\delta_0>0$,
$C=C(q)>0$ and $C^\prime=C^\prime(r)>0$
such that the steady problem
\eqref{NS-s} admits a unique solution
\begin{equation}
\begin{split}
&u_s\in L^q(\Omega)\cap C^\infty(\Omega), \quad
\|u_s\|_q\leq 
C|u_\infty|^{1/2},
\quad \forall q\in (2,\infty], \\
&\nabla u_s\in L^r(\Omega), \quad
\|\nabla u_s\|_r\leq 
C^\prime |u_\infty|^{1/2},
\quad \forall r\in (4/3,\infty], \\
&\mbox{provided $0<|u_\infty|\leq\delta_0$}.
\end{split}
\label{steady}
\end{equation}
Specifically, the rate $|u_\infty|^{1/2}$ above was deduced by Shibata 
as a consequence of his anisotropic pointwise estimates \cite[Theorem 1.1]{Shi}.
For the starting problem, we take this solution $u_s$.
For the landing problem, there is at least one solution to
\eqref{NS-s} having finite Dirichlet integral for every
$u_\infty\in\mathbb R^3\setminus\{0\}$ (see \cite{L1})
and, from now on, we fix a solution $u_s$;
then, it possesses the summability properties in \eqref{steady},
no matter which we may choose,
see Galdi \cite[Section X.6]{Ga-b}.

Given $u_0\in L^{3,\infty}_0(\Omega)$ with \eqref{compati},
we set
$v_0=u_0-h(0)u_s\in L^{3,\infty}_0(\Omega)$
which fulfills $\nu\cdot v_0|_{\partial\Omega}=0$ as well as
$\mbox{div $v_0$}=0$, see \eqref{NS-2}.
We take the extension $\bar v_0$ of $v_0$ by setting zero
outside $\Omega$; then, we have
$\bar v_0\in L^{3,\infty}_0(\mathbb R^3)$ with
$\mbox{div $\bar v_0$}=0$.
We fix $R>0$ such that
\begin{equation}
\mathbb R^3\setminus\Omega \subset B_R:=\{x\in\mathbb R^3; |x|<R\},
\label{cov}
\end{equation}
and take a cut-off function $\phi_0\in C_0^\infty(B_{2R})$
so that $\phi_0(x)=1$ in $B_R$.
Set
\begin{equation*}
\begin{split}
&\bar g(x,t)=(1-\phi_0(x))g(x,t), \\
&G(y,t)=\bar g\left(y+u_\infty\int_0^th(\tau)d\tau, t\right), \\
\end{split}
\end{equation*}
where $g$ is given by \eqref{force-1}.
Then it follows from \eqref{steady}
that $\bar g(t)$ belongs to $L^q(\mathbb R^3)\cap C^\infty(\mathbb R^3)$
for every $q\in (2,\infty]$ and, therefore, so does $G(t)$.
We also have
\[
\mbox{div $\bar g$}=
(1-\phi_0)(h-h^2)\sum_j(\partial_ju_s)\cdot\nabla u_{sj}
-g\cdot\nabla\phi_0,
\]
and, thereby, $\mbox{div $G(t)$}\in L^q(\mathbb R^3)$
for every $q\in [1,\infty]$,
which together with the Hardy-Littlewood-Sobolev inequality implies that
\[
Q(\cdot,t):=
\left(\frac{-1}{4\pi|\cdot|}*\mbox{div $G$}\right)(\cdot,t)
\in L^q(\mathbb R^3),\quad\forall q\in (3,\infty),
\]
where $*$ stands for the convolution on $\mathbb R^3$.
Set
\[
\mathbb P_{\mathbb R^3}G(t)=G(t)-\nabla Q(t)
\]
which satisfies
\begin{equation}
\begin{split}
\|\mathbb P_{\mathbb R^3}G(t)\|_{q,\mathbb R^3}
&\leq \|G(t)\|_{q,\mathbb R^3}+\|\nabla Q(t)\|_{q,\mathbb R^3} \\
&\leq C\|G(t)\|_{q,\mathbb R^3}
=C\|\bar g(t)\|_{q,\mathbb R^3}
\leq C\|g(t)\|_q
\leq CM_q
\end{split}
\label{G-est}
\end{equation}
for every $q\in (2,\infty)$ with
\begin{equation}
M_q=
|h^\prime|_\infty \|u_s\|_q +
(|h|_\infty+|h|_\infty^2)(\|u_s\|_\infty+|u_\infty|)\|\nabla u_s\|_q,
\label{M}
\end{equation}
where
\[
|h|_\infty=\sup_{t\geq 0}|h(t)|, \qquad |h^\prime|_\infty=\sup_{t\geq 0}|h^\prime(t)|.
\]
By using the heat semigroup
\[
e^{t\Delta}=(4\pi t)^{-3/2}e^{-|\cdot|^2/4t}* (\cdot),
\]
we set
\begin{equation}
\begin{split}
&V(t)=\int_0^te^{(t-\tau)\Delta}(\mathbb P_{\mathbb R^3}G)(\tau)d\tau, \\
&W(t)=e^{t\Delta}\bar v_0 +V(t).
\end{split}
\label{heat}
\end{equation}
Then the pair $W(y,t), Q(y,t)$ solves the Stokes initial value problem
\begin{equation}
\begin{split}
&\partial_tW=\Delta W-\nabla Q+G, \quad\mbox{div $W$}=0
\qquad (y\in\mathbb R^3,\,t>0), \\
&W\to 0 \quad\mbox{as $|y|\to\infty$}, \\
&W(y,0)=\bar v_0(y).
\end{split}
\label{W-eq}
\end{equation}
By \eqref{h-1} we know
\[
G\in C^\theta([0,\infty); L^q(\mathbb R^3)), \quad\forall q\in (2,\infty],
\]
which implies that
\begin{equation}
\begin{split}
&W\in C^1((0,\infty); L^{3,\infty}(\mathbb R^3)\cap L^q_\sigma(\mathbb R^3)),
\quad\forall q\in (3,\infty), \\
&\nabla^2 W\in 
C((0,\infty); L^{3,\infty}(\mathbb R^3)\cap L^q(\mathbb R^3)),
\quad\forall q\in (3,\infty).
\end{split}
\label{W-reg}
\end{equation}
We also find
\begin{equation}
\nabla W\in C^\mu_{loc}
((0,\infty); L^{3,\infty}(\mathbb R^3)\cap L^q(\mathbb R^3)),
\quad\forall q\in (3,\infty),\,\forall\mu\in (0,1/2).
\label{W-hoelder}
\end{equation}
We then make the change of variable as
\begin{equation}
\begin{split}
&U(x,t)=W\left(x-u_\infty\int_0^th(\tau)d\tau, t\right), \\
&P(x,t)=Q\left(x-u_\infty\int_0^th(\tau)d\tau, t\right),
\end{split}
\label{os-sol}
\end{equation}
to see from \eqref{W-reg}--\eqref{W-hoelder} that
\begin{equation}
\left\{
\begin{array}{ll}
U\in C^1((0,\infty); L^{3,\infty}(\mathbb R^3)\cap L^q_\sigma(\mathbb R^3)),
&\forall q\in (3,\infty), \\
\nabla^2U\in 
C((0,\infty); L^{3,\infty}(\mathbb R^3)\cap L^q(\mathbb R^3)),
&\forall q\in (3,\infty), \\
\nabla U\in C^\mu_{loc}
((0,\infty); L^{3,\infty}(\mathbb R^3)\cap L^q(\mathbb R^3)),
\quad &\forall q\in (3,\infty),\,\forall\mu\in (0,1/2),
\end{array}
\right.
\label{U-reg}
\end{equation}
and that the pair \eqref{os-sol} satisfies the non-autonomous
Oseen initial value problem
\begin{equation}
\begin{split}
&\partial_tU=\Delta U-\nabla P-hu_\infty\cdot\nabla U+\bar g,
\quad\mbox{div $U$}=0 \qquad 
(x\in\mathbb R^3,\,t>0), \\
&U\to 0 \quad\mbox{as $|x|\to\infty$}, \\
&U(x,0)=\bar v_0(x).
\end{split}
\label{oseen}
\end{equation}

Let us take another cut-off function $\phi\in C_0^\infty(B_{3R})$
so that $\phi(x)=1$ in $B_{2R}$.
Our auxiliary function is then given by
\begin{equation}
\widetilde U(x,t)=
(1-\phi(x))U(x,t)+\mathbb B[U(\cdot,t)\cdot\nabla\phi](x)
=U(x,t)+E(x,t),
\label{auxiliary}
\end{equation}
see \eqref{correc} below,
where $\mathbb B$ denotes the Bogovskii operator in the bounded domain
$A_R=B_{3R}\setminus\overline{B_R}$.
Since $\mbox{div $U$}=0$, we observe
$\int_{A_R}U\cdot\nabla\phi =0$,
which yields $\mbox{div $\widetilde U$}=0$.
By \eqref{U-reg} we find that
\begin{equation}
\left\{
\begin{array}{ll}
\widetilde U\in 
C^1((0,\infty); L^{3,\infty}(\Omega)\cap L^q_\sigma(\Omega)),
&\forall q\in (3,\infty), \\
\nabla^2\widetilde U\in
C((0,\infty); L^{3,\infty}(\Omega)\cap L^q(\Omega)),
&\forall q\in (3,\infty),   \\
\nabla\widetilde U\in C^\mu_{loc}
((0,\infty); L^{3,\infty}(\Omega)\cap L^q(\Omega)),
\quad &\forall q\in (3,\infty),\,\forall\mu\in (0,1/2),
\end{array}
\right.
\label{auxi-reg}
\end{equation}
and that
\begin{equation}
\begin{split}
&\partial_t\widetilde U=\Delta\widetilde U-\nabla P
-hu_\infty\cdot\nabla\widetilde U+g-F, \quad 
\mbox{div $\widetilde U$}=0 \qquad  (x\in\Omega,\,t>0),  \\
&\widetilde U|_{\partial\Omega}=0, \\
&\widetilde U\to 0 \quad\mbox{as $|x|\to\infty$}, \\
&\widetilde U(\cdot,0)=(1-\phi)v_0+\mathbb B[v_0\cdot\nabla\phi],
\end{split}
\label{auxi-eq}
\end{equation}
with
\begin{equation}
F(x,t):=\phi_0g-\partial_tE+\Delta E-hu_\infty\cdot\nabla E,
\label{auxi-force}
\end{equation}
where
\begin{equation}
E=-\phi U+\mathbb B[U\cdot\nabla\phi].
\label{correc}
\end{equation}

For later use, we collect some properties of $U$ and $\widetilde U$.
\begin{lemma}
Let $j=0,1$.
The function $U$ given by \eqref{os-sol} enjoys
\begin{equation}
\begin{split}
&\|U(t)\|_{\infty,\mathbb R^3}
\leq C(\|v_0\|_{3,\infty}+M_3)\,t^{-1/2}, \\
&\|\nabla^j U(t)\|_{r,\mathbb R^3}
\leq C(\|v_0\|_{3,\infty}+M_3)\,t^{-1/2+3/2r-j/2}, \quad
\forall r\in (3,\infty), \\
&\|\nabla^j U(t)\|_{3,\infty,\mathbb R^3}
\leq C(\|v_0\|_{3,\infty}+M_3)\,t^{-j/2}
\end{split}
\label{U-est}
\end{equation}
for all $t>0$,
where $M_3$ is as in \eqref{M}, and
\begin{equation}
\begin{split}
&\|U(t)\|_{r,\mathbb R^3}=o(t^{-1/2+3/2r}), \quad\forall r\in (3,\infty],\\
&\|U(t)\|_{3,\infty,\mathbb R^3}=o(1),
\end{split}
\label{U-decay}
\end{equation}
as $t\to\infty$.
\label{auxi-lem}
\end{lemma}
\begin{proof}
Since
\[
\|\nabla^jU(t)\|_{r,\mathbb R^3}=\|\nabla^jW(t)\|_{r,\mathbb R^3}, \qquad
\|\nabla^jU(t)\|_{3,\infty,\mathbb R^3}=\|\nabla^jW(t)\|_{3,\infty,\mathbb R^3},
\]
it suffices to show the desired properties for $W(t)$ given by \eqref{heat}.
By the Hausdorff-Young inequality and by real interpolation,
we easily see that
\[
\|\nabla^je^{t\Delta}\bar v_0\|_{r,\mathbb R^3}
\leq Ct^{-1/2+3/2r-j/2}\|v_0\|_{3,\infty}, \quad
\|\nabla^je^{t\Delta}\bar v_0\|_{3,\infty,\mathbb R^3}
\leq Ct^{-j/2}\|v_0\|_{3,\infty},
\]
for $3<r\leq\infty$.
We use the assumption \eqref{h-2} and \eqref{G-est} with $q=3$ to observe
\begin{equation}
\|\nabla^jV(t)\|_{r,\mathbb R^3}\leq CM_3\int_0^{T_0}(t-\tau)^{-1/2+3/2r-j/2}d\tau
\leq CM_3T_0\,t^{-1/2+3/2r-j/2}
\label{V-large}
\end{equation}
for $t\geq 2T_0$, while
\begin{equation}
\|\nabla^jV(t)\|_{r,\mathbb R^3}\leq CM_3T_0^{1/2+3/2r-j/2}
\label{V-short1}
\end{equation}
for $t\in (0,2T_0]$ (except for the case $(j,r)=(1,\infty)$).
Similarly, we obtain
\[
\|\nabla^jV(t)\|_{3,\infty,\mathbb R^3}\leq CM_3T_0t^{-j/2}
\]
for $t\geq 2T_0$ and
\[
\|\nabla^jV(t)\|_{3,\infty,\mathbb R^3}\leq CM_3T_0^{1-j/2}
\]
for $t\in (0,2T_0]$.
This shows \eqref{U-est}.

The sharp behavior \eqref{U-decay} was observed 
by \cite{Ma13}, but let us give the proof for completeness.
For $v_0\in L^{3,\infty}_0(\Omega)$ and
every $\varepsilon >0$, one can take
$v_{0\varepsilon}\in C^\infty_0(\Omega)\subset
C^\infty_0(\mathbb R^3)$ such that
\begin{equation}
\|v_{0\varepsilon}-\bar v_0\|_{3,\infty,\mathbb R^3}
=\|v_{0\varepsilon}-v_0\|_{3,\infty}\leq\varepsilon.
\label{appro}
\end{equation}
Then we have
\[
\|e^{t\Delta}\bar v_0\|_{3,\infty,\mathbb R^3}
\leq C\|v_{0\varepsilon}\|_{1,\mathbb R^3}\,t^{-1}+C\varepsilon,
\]
yielding
$\limsup_{t\to\infty}\|e^{t\Delta}\bar v_0\|_{3,\infty,\mathbb R^3}
\leq C\varepsilon$, which also implies
\[
\|e^{t\Delta}\bar v_0\|_{r,\mathbb R^3}
\leq Ct^{-1/2+3/2r}\|e^{\frac{t}{2}\Delta}\bar v_0\|_{3,\infty,\mathbb R^3}
=o(t^{-1/2+3/2r})
\]
as $t\to\infty$.
In \eqref{V-large} one can use \eqref{G-est} with $p\in (2,3)$
to replace $M_3$ by $M_p$; then,
\begin{equation*}
\begin{split}
&\|V(t)\|_{r,\mathbb R^3}
\leq CM_p T_0\,t^{-3/2p+3/2r},  \\
&\|V(t)\|_{3,\infty,\mathbb R^3}
\leq CM_p T_0\,t^{-3/2p+1/2},
\end{split}
\end{equation*}
for $t\geq 2T_0$, which proves \eqref{U-decay}.
\end{proof}
\begin{corollary}
Let $j=0,1$.
The function $\widetilde U$ given by \eqref{auxiliary} enjoys
\begin{equation}
\|\widetilde U(t)\|_r
\leq C(\|v_0\|_{3,\infty}+M_3)\,t^{-1/2+3/2r}, \quad\forall r\in (3,\infty],
\label{auxi-est-1}
\end{equation}
\begin{equation}
\|\nabla\widetilde U(t)\|_r
\leq C(\|v_0\|_{3,\infty}+M_3)\,t^{-1+3/2r}(1+t)^{1/2-3/2r},\quad
\forall r\in (3,\infty),
\label{auxi-est-3}
\end{equation}
\begin{equation}
\|\nabla^j \widetilde U(t)\|_{3,\infty}
\leq C(\|v_0\|_{3,\infty}+M_3)\,t^{-j/2},
\label{auxi-est-2}
\end{equation}
for all $t>0$,
where $M_3$ is as in \eqref{M}, and
\begin{equation}
\begin{split}
&\|\widetilde U(t)\|_r=o(t^{-1/2+3/2r}), \quad\forall r\in (3,\infty], \\
&\|\widetilde U(t)\|_{3,\infty}=o(1),
\end{split}
\label{auxi-decay}
\end{equation}
as $t\to\infty$.

Let $\bar t\in [T_0,\infty)$, where $T_0$ is as in \eqref{h-2} or \eqref{h-3},
then
\begin{equation}
\begin{split}
&\|\widetilde U(t)\|_r
\leq C(t-\bar t)^{-1/2+3/2r}
\|U(\bar t)\|_{3,\infty,\mathbb R^3}, \quad \forall r\in (3,\infty], \\
&\|\nabla \widetilde U(t)\|_{3,\infty}
\leq C(t-\bar t)^{-1/2}
\|U(\bar t)\|_{3,\infty,\mathbb R^3},
\end{split}
\label{auxi-bar}
\end{equation}
for all $t>\bar t$.
\label{auxi-cor}
\end{corollary}
\begin{proof}
On account of \eqref{bog-est-1}
(combined with the Gagliardo-Nirenberg inequality for $r=\infty$) we have
\begin{equation*}
\begin{split}
&\|\widetilde U(t)\|_r\leq C\|U(t)\|_{r,\mathbb R^3}, \\
&\|\nabla\widetilde U(t)\|_r
\leq C\|\nabla U(t)\|_{r,\mathbb R^3}+C\|U(t)\|_{\infty,A_R}.
\end{split}
\end{equation*}
for $r\in (3,\infty]$ as well as the similar inequalities for
$\|\nabla^j \widetilde U(t)\|_{3,\infty}$, see \eqref{bog-est-L}.
Then Lemma \ref{auxi-lem} concludes \eqref{auxi-est-1}, \eqref{auxi-est-2},
\eqref{auxi-est-3} and \eqref{auxi-decay}.

By \eqref{h-2} or \eqref{h-3} we have
$G(y,t)=0$ for $t\geq T_0$ and, therefore, deduce from \eqref{W-eq} that
$W(t)=e^{(t-\bar t)\Delta}W(\bar t)$.
In view of \eqref{os-sol} and \eqref{auxiliary} we find
\begin{equation*}
\|\widetilde U(t)\|_r\leq C\|W(t)\|_{r,\mathbb R^3}
\leq C(t-\bar t)^{-1/2+3/2r}\|W(\bar t)\|_{3,\infty,\mathbb R^3},
\end{equation*}
for $3<r\leq\infty$.
Similarly, we have
\[
\|\nabla\widetilde U(t)\|_{3,\infty}
\leq C\|\nabla W(t)\|_{3,\infty,\mathbb R^3}
+C\|W(t)\|_{\infty,\mathbb R^3}
\leq C(t-\bar t)^{-1/2}\|W(\bar t)\|_{3,\infty,\mathbb R^3}.
\]
These estimates together with 
$\|W(\bar t)\|_{3,\infty,\mathbb R^3}
=\|U(\bar t)\|_{3,\infty,\mathbb R^3}$
imply \eqref{auxi-bar}.
\end{proof}
\begin{remark}
Actually, $\widetilde U(t)$ does not possess any singular behavior
near $t=\bar t$, however, it is convenient to use \eqref{auxi-bar}
in the proof of Proposition \ref{mild}.
\label{conve-1}
\end{remark}

We will be faced with some troubles a few times
arising from
the behavior of $\widetilde U(t)$ such as
$\|\widetilde U(t)\|_\infty^2\leq Ct^{-1}$
near $t=0$, see \eqref{auxi-est-1}.
In order to get around this
unpleasant
situation, it is convenient to
carry out the following simple approximation procedure.
\begin{lemma}
Let $\varepsilon >0$.
Then there is a function
\[
\widetilde U_\varepsilon\in 
L^\infty(0,\infty; L^q(\Omega))
\]
with
\[
\nabla \widetilde U_\varepsilon\in 
L^\infty(0,\infty; L^q(\Omega))
\]
for every $q\in (3,\infty]$
such that
\begin{equation*}
\begin{split}
&\sup_{t>0}\|\widetilde U_\varepsilon(t)-\widetilde U(t)\|_{3,\infty}
\leq C\varepsilon,\\
&\sup_{t>0}t^{1/2-3/2q}\|\widetilde U_\varepsilon(t)-\widetilde U(t)\|_q
\leq C\varepsilon, \\
&\sup_{0<t\leq 1}t^{1-3/2q}
\|\nabla\widetilde U_\varepsilon(t)-\nabla\widetilde U(t)\|_q
\leq C\varepsilon,
\end{split}
\end{equation*}
for every $q\in (3,\infty)$.
\label{auxi-appro}
\end{lemma}
\begin{proof}
We use $v_{0\varepsilon}$ in \eqref{appro}.
We replace $\bar v_0$ by $v_{0\varepsilon}$
in \eqref{heat} to define $W_\varepsilon$,
which leads to $\widetilde U_\varepsilon$ by \eqref{auxiliary} 
via \eqref{os-sol}.
Then we have
\begin{equation*}
\begin{split}
&\|\widetilde U_\varepsilon(t)-\widetilde U(t)\|_q
\leq C\|W_\varepsilon(t)-W(t)\|_{q,\mathbb R^3} 
=C\|e^{t\Delta}(v_{0\varepsilon}-\bar v_0)\|_{q,\mathbb R^3}, \\
&\|\widetilde U_\varepsilon(t)-\widetilde U(t)\|_{3,\infty}
\leq C\|W_\varepsilon(t)-W(t)\|_{3,\infty,\mathbb R^3}
=C\|e^{t\Delta}(v_{0\varepsilon}-\bar v_0)\|_{3,\infty,\mathbb R^3},
\end{split}
\end{equation*}
and
\begin{equation*}
\begin{split}
\|\nabla\widetilde U_\varepsilon(t)-\nabla\widetilde U(t)\|_q
&\leq C\|\nabla W_\varepsilon(t)-\nabla W(t)\|_{q,\mathbb R^3}
+C\|W_\varepsilon(t)-W(t)\|_{q,A_R} \\
&=C\|\nabla e^{t\Delta}(v_{0\varepsilon}-\bar v_0)\|_{q,\mathbb R^3}
+C\|e^{t\Delta}(v_{0\varepsilon}-\bar v_0)\|_{q,\mathbb R^3},
\end{split}
\end{equation*}
as well as
\begin{equation*}
\begin{split}
\|\widetilde U_\varepsilon(t)\|_q
&\leq C\|W_\varepsilon(t)\|_{q,\mathbb R^3}
\leq C\|e^{t\Delta}v_{0\varepsilon}\|_{q,\mathbb R^3}
+C\|V(t)\|_{q,\mathbb R^3}, \\
\|\nabla\widetilde U_\varepsilon(t)\|_q
&\leq C\|\nabla W_\varepsilon(t)\|_{q,\mathbb R^3}
+C\|W_\varepsilon(t)\|_{q,A_R} \\
&\leq C\|e^{t\Delta}\nabla v_{0\varepsilon}\|_{q,\mathbb R^3}
+C\|e^{t\Delta}v_{0\varepsilon}\|_{q,\mathbb R^3}
+C\|V(t)\|_{W^{1,q}(\mathbb R^3)},
\end{split}
\end{equation*}
for every $q\in (3,\infty]$.
Concerning $\|\nabla ^jV(t)\|_{q,\mathbb R^3}$ for $j=0,1$,
we have \eqref{V-large} and \eqref{V-short1} except for the case
$(j,q)=(1,\infty)$, in which $\|\nabla V(t)\|_{\infty,\mathbb R^3}$
can be estimated similarly by use of
\eqref{G-est} with $q\in (3,\infty)$.
The proof is thus complete.
\end{proof}
\begin{remark}
Both $\widetilde U_\varepsilon$ and $\nabla \widetilde U_\varepsilon$
belong to 
$L^\infty(0,\infty; L^q(\Omega))$ for every $q>2$
since we have \eqref{G-est} for such $q$, however,
for later use, the only cases we need are $q=\infty$ and $q=6$.
\label{rem-aapro}
\end{remark}

We next deduce some estimates and regularity of the function $F$.
\begin{lemma}
The function $F$ given by \eqref{auxi-force} satisfies
\begin{equation}
\|F(t)\|_2\leq
C(\|v_0\|_{3,\infty}+M_3)\,t^{-1/2},
\label{F-est-1}
\end{equation}
\begin{equation}
\|F(t)\|_{H^{-1}(\Omega)}\leq 
C(\|v_0\|_{3,\infty}+M_3)(1+t)^{-1/2},
\label{F-est-2}
\end{equation}
\begin{equation}
|\langle F(t), \varphi\rangle|
\leq C(\|v_0\|_{3,\infty}+M_3)(1+t)^{-1/2}\|\nabla\varphi\|_2, \quad
\forall \varphi\in H_0^1(\Omega),
\label{F-est-3}
\end{equation}
for all $t>0$, where $M_3$ is as in 
\eqref{M}, and thereby
\begin{equation}
F\in L^2(0,T; H^{-1}(\Omega))
\label{F}
\end{equation}
for every $T\in (0,\infty)$.
Furthermore,
\begin{equation}
F\in C^\mu_{loc}((0,\infty); L^2(\Omega))
\label{F-hoelder}
\end{equation}
for every $\mu\in (0,1/2)$ with $\mu\leq\theta$,
where $\theta$ is as in \eqref{h-1}.

Let $q\in (1,3)$ and $\bar t\in [T_0,\infty)$, where $T_0$ is as in
\eqref{h-2} or \eqref{h-3}.
Then
\begin{equation}
\|F(t)\|_q
\leq C(t-\bar t)^{-1/2}\|U(\bar t)\|_{3,\infty,\mathbb R^3}
\label{F-bar}
\end{equation}
for all $t>\bar t$.
\label{F-class}
\end{lemma}
\begin{proof}
Using the equation \eqref{oseen}, we split $F$ into
\[
F(x,t)=F_1+F_2+F_3+F_4+F_5
\]
with
\begin{equation*}
\begin{split}
&F_1=\phi(g-\nabla P)-\mathbb B[(g-\nabla P)\cdot\nabla\phi], \\
&F_2=-2\nabla\phi\cdot\nabla U, \\
&F_3=-(\Delta\phi)U+h(u_\infty\cdot\nabla\phi)U
-hu_\infty\cdot\nabla\mathbb B[U\cdot\nabla\phi], \\
&F_4=h\mathbb B[(u_\infty\cdot\nabla U)\cdot\nabla\phi], \\
&F_5=-\mathbb B[\Delta U\cdot\nabla\phi]+\Delta\mathbb B[U\cdot\nabla\phi].
\end{split}
\end{equation*}
Here, we have used $\phi_0g+\phi\bar g=\phi g$.
It is easily seen from \eqref{G-est} that
\[
\|F_1\|_2
\leq C\|g(t)\|_3+C\|\nabla Q(t)\|_{3,\mathbb R^3}\leq CM_3.
\]
Note that
\[
F_1=0 \qquad (t\geq T_0)
\]
by \eqref{h-2} or \eqref{h-3}.
We also have
\begin{equation*}
\begin{split}
&\|F_2\|_2\leq C\|\nabla U(t)\|_{2,A_R}\leq C\|\nabla U(t)\|_{3,\infty,\mathbb R^3}, \\
&\|F_2\|_{H^{-1}(\Omega)}+\|F_3\|_2
\leq C\|U(t)\|_{2,A_R}.
\end{split}
\end{equation*}
Thanks to \eqref{bog-est-2}, we obtain
\[
\|F_4\|_2\leq C\|(u_\infty\cdot\nabla U)\cdot\nabla\phi\|_{H^1(A_R)^*}
\leq C\|U(t)\|_{2,A_R}.
\]
The last term is further modified as
\[
F_5=F_{51}+F_{52},
\]
where
\begin{equation*}
\begin{split}
&F_{51}=
-\mathbb B[\Delta(U\cdot\nabla\phi)]+\Delta\mathbb B[U\cdot\nabla\phi], \\
&F_{52}=
\mathbb B[2\nabla U\cdot\nabla(\nabla\phi)+U\cdot\nabla(\Delta\phi)].
\end{split}
\end{equation*}
From \eqref{bog-est-2} as well as \eqref{bog-est-1} we observe
\[
\|F_{52}\|_2\leq C\|U(t)\|_{2,A_R}.
\]
By virtue of \eqref{bog-est-3} we find
\[
\|F_{51}\|_2
\leq C\|U(t)\|_{2,A_R}.
\]
All the computation above tells us that
\[
|\langle F(t),\varphi\rangle|
\leq \|F_1+F_3+F_4+F_5\|_2\|\varphi\|_{2,\Omega_{3R}}
+C\|U(t)\|_{2,A_R}\|\varphi\|_{H^1(\Omega_{3R})},
\]
the latter of which comes from $F_2$,
where $\Omega_{3R}=\Omega\cap B_{3R}$.
Since
$\|\varphi\|_{2,\Omega_{3R}}\leq C\|\nabla\varphi\|_2$
for $\varphi\in H_0^1(\Omega)$, we get
\[
|\langle F(t), \varphi\rangle|
\leq C\|U(t)\|_{2,A_R}\|\nabla\varphi\|_2.
\]
Using
\[
\|U(t)\|_{2,A_R}\leq 
\left\{
\begin{array}{l}
C\|U(t)\|_{3,\infty,\mathbb R^3}, \\
C\|U(t)\|_{\infty,\mathbb R^3},
\end{array}
\right.
\]
we conclude \eqref{F-est-1}--\eqref{F-est-3} from \eqref{U-est}.

Estimates above in $L^2(\Omega)$ imply that
\begin{equation*}
\begin{split}
\|F(t)-F(s)\|_2
&\leq C\|g(t)-g(s)\|_3
+C\|\nabla U(t)-\nabla U(s)\|_{2,A_R}  \\
& \qquad +C\|U(t)-U(s)\|_{2,A_R}
+C|h(t)-h(s)|,
\end{split}
\end{equation*}
which leads us to \eqref{F-hoelder} on account of
\eqref{h-1}, \eqref{force-1} and \eqref{U-reg}.

Finally, let $q\in (1,3)$, $\bar t\in [T_0,\infty)$ and $t>\bar t$.
Since estimates above in $L^2(\Omega)$ replaced by $L^q(\Omega)$ hold true,
we have
\begin{equation*}
\|F(t)\|_q\leq C\|\nabla U(t)\|_{q,A_R}+C\|U(t)\|_{q,A_R}
\leq C\|\nabla U(t)\|_{3,\infty,\mathbb R^3}+C\|U(t)\|_{\infty,\mathbb R^3}.
\end{equation*}
Then the same reasoning as in the proof of \eqref{auxi-bar}
yields \eqref{F-bar}.
\end{proof}

\section{Weak solution}\label{weak}

Let us take the auxiliary function $\widetilde U(x,t)$ given by \eqref{auxiliary}
and look for a solution to \eqref{NS-1} of the form \eqref{decompo}.
Then \eqref{NS-2} and \eqref{auxi-eq} imply that $w(x,t)$ should obey
\eqref{NS-3} with
\begin{equation}
\begin{split}
&f=F-\widetilde U\cdot\nabla\widetilde U
-h(u_s\cdot\nabla\widetilde U+\widetilde U\cdot\nabla u_s), \\
&w_0=\phi v_0-\mathbb B[v_0\cdot\nabla\phi]\in L^2_\sigma(\Omega),
\end{split}
\label{data}
\end{equation}
where $p_w=p_v-P$ is the pressure associated with $w$, while $F$ is given by 
\eqref{auxi-force}.
By \eqref{interpo-1}, \eqref{LH}, \eqref{LS}, \eqref{auxi-est-1}
and \eqref{auxi-est-2} we have
\begin{equation*}
\begin{split}
&\|\widetilde U\cdot\nabla\widetilde U\|_2
\leq\|\widetilde U\|_{6,2}\|\nabla \widetilde U\|_{3,\infty}
\leq C(\|v_0\|_{3,\infty}+M_3)^2\,t^{-3/4}, \\
&\|u_s\cdot\nabla\widetilde U\|_2
\leq C\|u_s\|_{6,2}(\|v_0\|_{3,\infty}+M_3)\,t^{-1/2}, \\
&\|\widetilde U\cdot\nabla u_s\|_2
\leq C(\|\nabla u_s\|_{6,2}+\|\nabla u_s\|_2)(\|v_0\|_{3,\infty}+M_3)(1+t)^{-1/2},
\end{split}
\end{equation*}
for all $t>0$.
These estimates together with \eqref{F-est-1} imply
\begin{equation}
\kappa_f:=\sup_{t>0}t^{3/4}(1+t)^{-1/4}\|f(t)\|_2 <\infty.
\label{f-est}
\end{equation}
By \eqref{auxi-est-1} and \eqref{auxi-decay} we know
\begin{equation}
\|\widetilde U\otimes \widetilde U+h(\widetilde U\otimes u_s+u_s\otimes \widetilde U)\|_2
\left\{
\begin{array}{ll}
\leq Ct^{-1/4} \quad &\mbox{for all $t>0$}, \\
=o(t^{-1/4}) &\mbox{as $t\to\infty$},
\end{array}
\right.
\label{tensor-est}
\end{equation}
which together with \eqref{F} yields
\begin{equation}
f\in L^2(0,T; H^{-1}(\Omega))
\label{f-class}
\end{equation}
for every $T\in (0,\infty)$.
Furthermore, by \eqref{h-1}, \eqref{auxi-reg} and \eqref{F-hoelder} we find
\begin{equation}
f\in C^\mu_{loc}((0,\infty); L^2(\Omega)),
\label{f-hoelder}
\end{equation}
for every $\mu\in (0,1/2)$ with $\mu\leq\theta$.

In this section we show the existence of weak solution with the strong energy 
inequality \eqref{SEI}.
Let us recall the definition of 
the Leray-Hopf weak solution (\cite{L2}, \cite{Ho}, \cite{Mas}).
\begin{definition}
We say that $w(x,t)$ is a weak solution to \eqref{NS-3} with \eqref{data} if
\[
w\in L^\infty(0,T; L^2_\sigma(\Omega))
\cap L^2(0,T; H^1_{0,\sigma}(\Omega))
\cap C_w([0,\infty); L^2_\sigma(\Omega))
\]
for all $T\in (0,\infty)$ together with
$\lim_{t\to 0}\|w(t)-w_0\|_2=0$
and $w$ satisfies \eqref{SEI} for $s=0$ as well as
\begin{equation}
\begin{split}
&\langle w(t),\varphi(t)\rangle
+\int_s^t
\Big[\langle\nabla w,\nabla\varphi\rangle
+\langle\{h(u_\infty+u_s)+\widetilde U\}\cdot\nabla w,\varphi\rangle \\
&\qquad -\langle (hu_s+\widetilde U)\otimes w, \nabla\varphi\rangle 
+\langle w\cdot\nabla w,\varphi\rangle\Big] d\tau \\
&=\langle w(s),\varphi(s)\rangle
+\int_s^t \Big[\langle w,\partial_\tau\varphi\rangle
+ \langle f,\varphi\rangle \Big] d\tau
\end{split}
\label{weak-form}
\end{equation}
for all $0\leq s<t<\infty$ and
$\varphi$, which is of class
\begin{equation}
\begin{split}
&\varphi\in C([0,\infty); L^2_\sigma(\Omega))
\cap L^\infty_{loc}([0,\infty); L^{3,\infty}(\Omega)), \\
&\nabla\varphi\in L^2_{loc}([0,\infty); L^2(\Omega)), \quad
\partial_t\varphi\in L^2_{loc}([0,\infty); L^2_\sigma(\Omega)).
\end{split}
\label{test}
\end{equation}
\label{Leray-Hopf}
\end{definition}

We will follow in principle the argument of Miyakawa and Sohr \cite{MS},
whose idea partially goes back to Leray \cite{L2}.
Set
\[
J_k=e^{-\frac{1}{k}A}, \qquad (k=1,2,...)
\]
and consider the approximate problem
\begin{equation}
\begin{split}
&\partial_tw+Aw+\mathbb P[Sw+(J_kw)\cdot\nabla w]=\mathbb Pf, \\
&w(0)=w_0,
\end{split}
\label{NS-appro}
\end{equation}
where
\[
Sw=\{h(u_\infty+u_s)+\widetilde U\}\cdot\nabla w+w\cdot\nabla (hu_s+\widetilde U).
\]
The following lemma provides a solution with the a priori estimate.
\begin{lemma}
For each $k=1,2,...$, problem \eqref{NS-appro} admits a unique global strong
solution $w=w_k$ of class
\[
w_k\in C([0,\infty); L^2_\sigma(\Omega))
\cap C((0,\infty); D_2(A))\cap C^1((0,\infty); L^2_\sigma(\Omega))
\]
subject to
$\lim_{t\to 0}\|w_k(t)-w_0\|_2=0$, which satisfies
\begin{equation}
\|w_k(t)\|_2^2+\int_0^t\|\nabla w_k\|_2^2d\tau\leq Y(t)
\label{appro-est}
\end{equation}
for all $t>0$ with
\begin{equation}
\begin{split}
&Y(t):=\left(\|w_0\|_2^2+C\|f\|_{L^2(0,t; H^{-1}(\Omega))}^2\right)e^{CNt}, \\
&N:=1+|h|_\infty^2\|u_s\|_\infty^2
+\|\widetilde U_{\varepsilon_0}\|_{L^\infty(0,\infty; L^\infty(\Omega))}^2,
\end{split}
\label{Y}
\end{equation}
where $\widetilde U_{\varepsilon_0}$ is the function given by
Lemma \ref{auxi-appro} for some $\varepsilon_0 >0$.
\label{appro-sol}
\end{lemma}
\begin{proof}
We fix $T\in (1,\infty)$ arbitrarily, and let us construct a solution on $(0,T]$.
We first establish the local existence of solutions.
Let $T_*\in (0,1]$ and set
\begin{equation*}
\begin{split}
E_{T_*}=&\{w\in C((0,T_*]; H^1_{0,\sigma}(\Omega)); \\
&\quad \|w\|_{E_{T_*}}:=
\sup_{0<t\leq T_*}\left(\|w(t)\|_2+t^{1/2}\|\nabla w(t)\|_2\right)<\infty\}
\end{split}
\end{equation*}
which is a Banach space endowed with norm $\|\cdot\|_{E_{T_*}}$.
We set
\[
(\Phi w)(t)
=H(t)-\int_0^t e^{-(t-\tau)A}\mathbb P[Sw+(J_kw)\cdot\nabla w](\tau)d\tau,
\]
where
\[
H(t)=e^{-tA}w_0+\int_0^t e^{-(t-\tau)A}\mathbb Pf(\tau)d\tau,
\]
and intend to solve the integral equation 
$w=\Phi w$ in $E_{T_*}$
by using \eqref{d-1}--\eqref{d-2} (for the Stokes semigroup).
For $w\in E_{T_*}$, we easily find
\begin{equation}
\begin{split}
&\Phi w\in C^{\mu}_{loc}((0,T_*]; L^q_\sigma(\Omega)), \quad
\forall q\in [2,\infty),\, \forall \mu\in (0,\mu_0), \\
&\nabla\Phi w\in C^{\mu}_{loc}((0,T_*]; L^q(\Omega)), \quad
\forall q\in [2,6),\, \forall \mu\in (0,\mu_0-1/2),
\end{split}
\label{Phi}
\end{equation}
where $\mu_0=\frac{3}{2q}+\frac{1}{4}$.
By \eqref{f-est} we have $H\in E_{T_*}$ with
\begin{equation*}
\begin{split}
&\|H(t)-w_0\|_2\leq\|e^{-tA}w_0-w_0\|_2+C\kappa_f\, t^{1/4}(1+t)^{1/4}, \\
&\|H\|_{E_{T_*}}\leq C_0\left(\|w_0\|_2+\kappa_f\sqrt{T}\right).
\end{split}
\end{equation*}
Let $w\in E_{T_*}$, then we have
\begin{equation*}
\begin{split}
\|\nabla^j\int_0^t e^{-(t-\tau)A}\mathbb P[(J_kw)\cdot\nabla w](\tau)d\tau\|_2
&\leq C\int_0^t (t-\tau)^{-j/2}\|J_kw\|_\infty\|\nabla w\|_2d\tau \\
&\leq C_1k^{3/4}\sqrt{T_*}\,t^{-j/2}\|w\|_{E_{T_*}}^2
\end{split}
\end{equation*}
for $t\in (0,T_*]$ 
and $j=0,1$.
Let $\varepsilon >0$.
We fix $r\in (3,\infty)$ and employ $\widetilde U_\varepsilon$ in
Lemma \ref{auxi-appro} to find
\begin{equation*}
\begin{split}
&\quad \|\nabla^j\int_0^t e^{-(t-\tau)A}\mathbb PSw(\tau)d\tau\|_2 \\
&\leq C\int_0^t(t-\tau)^{-j/2}
\left(|h|_\infty\|u_\infty+u_s\|_\infty+\|\widetilde U_\varepsilon\|_\infty\right)
\|\nabla w\|_2d\tau \\
&\quad +C\int_0^t (t-\tau)^{-3/2r-j/2}
\|\widetilde U_\varepsilon-\widetilde U\|_r\|\nabla w\|_2 d\tau \\
&\quad +C\int_0^t (t-\tau)^{-j/2}
\left(|h|_\infty\|\nabla u_s\|_\infty+\|\nabla\widetilde U_\varepsilon\|_\infty\right)
\|w\|_2d\tau \\
&\quad +C\int_0^t (t-\tau)^{-3/2r-j/2}
\|\nabla\widetilde U_\varepsilon-\nabla\widetilde U\|_r\|w\|_2 d\tau \\
&\leq 
\big\{
C_2^{(\varepsilon)}(\sqrt{T_*}+T_*)+
C_2^\prime\varepsilon
\big\}\,t^{-j/2}
\|w\|_{E_{T_*}}
\end{split}
\end{equation*}
for $t\in (0,T_*]$ and $j=0,1$.
As a consequence, we obtain
\begin{equation*}
\begin{split}
\|\Phi w\|_{E_{T_*}}
&\leq C_0\left(\|w_0\|_2+\kappa_f\sqrt{T}\right)
+C_1k^{3/4}\sqrt{T_*}\|w\|_{E_{T_*}}^2  \\
&\qquad +
\big(2C_2^{(\varepsilon)}\sqrt{T_*}+C_2^\prime \varepsilon\big)
\|w\|_{E_{T_*}}
\end{split}
\end{equation*}
as well as
\[
\limsup_{t\to 0}\|(\Phi w)(t)-w_0\|_2\leq C\varepsilon\|w\|_{E_{T_*}}
\]
for $w\in E_{T_*}$.
The latter for arbitrary $\varepsilon >0$ yields
\begin{equation}
\lim_{t\to 0}\|(\Phi w)(t)-w_0\|_2=0.
\label{IC-appro}
\end{equation}
We next choose $\varepsilon =1/8C_2^\prime$ in the former, so that
$2C_2^{(\varepsilon)}\sqrt{T_*}+C_2^\prime \varepsilon \leq 1/4$ when
$T_*\leq (1/16C_2^{(\varepsilon)})^2$.
We set 
\[
E_{T_*,\rho}=\{w\in E_{T_*}; \|w\|_{E_{T_*}}\leq\rho\}
\]
with
\begin{equation}
\rho=2C_0\left(\|w_0\|_2+\kappa_f\sqrt{T}\right), \quad
T_*=\min\Big\{\left(4C_1k^{3/4}\rho\right)^{-2},\,
(16C_2^{(\varepsilon)})^{-2},\, 1\Big\}.
\label{length}
\end{equation}
Then $w\in E_{T_*,\rho}$ implies $\Phi w\in E_{T_*,\rho}$.
Furthermore, we find
\[
\|\Phi w_1-\Phi w_2\|_{E_{T_*}}\leq\frac{3}{4}\|w_1-w_2\|_{E_{T_*}}
\]
for $w_1, w_2\in E_{T_*,\rho}$.
We thus get a unique fixed point $w\in E_{T_*,\rho}$ of the map $\Phi$,
which fulfills the initial condition by \eqref{IC-appro}.
It also follows from \eqref{Phi} together with 
\eqref{h-1}, \eqref{auxi-reg} and \eqref{f-hoelder} that
the local solution $w(t)$ satisfies
\[
\mathbb P[f-Sw-(J_kw)\cdot\nabla w]\in C^\mu_{loc}((0,T_*]; L^2_\sigma(\Omega)), \quad
\forall\mu\in (0,1/2)\;\mbox{with $\mu\leq\theta$}.
\]
Therefore, $w(t)$ is a strong solution of class
\[
w\in
C([0,T_*]; L^2_\sigma(\Omega))
\cap C((0,T_*]; D_2(A))
\cap C^1((0,T_*]; L^2_\sigma(\Omega)).
\]

In view of \eqref{length}, it suffices to derive a priori estimate
of strong solutions in $L^2(\Omega)$ for continuation of 
the solution globally in time.
Let $\varepsilon >0$.
By \eqref{NS-appro} we have
\begin{equation}
\frac{1}{2}\frac{d}{dt}\|w(t)\|_2^2+\|\nabla w(t)\|_2^2
=\langle (hu_s+\widetilde U)\otimes w, \nabla w\rangle+\langle f,w\rangle.
\label{EE}
\end{equation}
We use Lemma \ref{auxi-appro} again to find that it is bounded from above by
\begin{equation*}
\begin{split}
&C\|f(t)\|_{H^{-1}(\Omega)}^2
+C\left(1+|h(t)|^2\|u_s\|_\infty^2+\|\widetilde U_\varepsilon(t)\|_\infty^2\right)
\|w(t)\|_2^2  \\
&+\frac{1}{4}\|\nabla w(t)\|_2^2
+C_3\|\widetilde U_\varepsilon(t)-\widetilde U(t)\|_{3,\infty}\|\nabla w(t)\|_2^2.
\end{split}
\end{equation*}
We choose $\varepsilon=\varepsilon_0$ 
such that $\sup_{t>0}\|\widetilde U_{\varepsilon_0}(t)-\widetilde U(t)\|_{3,\infty}\leq 1/4C_3$
to conclude \eqref{appro-est}.
\end{proof}

Let $T\in (0,\infty)$.
By \eqref{appro-est} one can find a subsequence of
$\{w_k\}$, which is denoted by itself,
as well as a function
\begin{equation}
w\in L^\infty(0,T; L^2_\sigma(\Omega))
\cap L^2(0,T; H^1_{0,\sigma}(\Omega))
\label{candi}
\end{equation}
so that
\begin{equation}
\begin{split}
&w_k\to w \quad\mbox{weakly-star in $L^\infty(0,T; L^2_\sigma(\Omega))$}, \\
&w_k\to w \quad\mbox{weakly in $L^2(0,T; H^1_{0,\sigma}(\Omega))$},
\end{split}
\label{conv-1}
\end{equation}
as $k\to\infty$.
Let us deduce further convergence of $\{w_k\}$
\begin{lemma}
Let $T\in (0,\infty)$, and let $w$ be the function obtained in \eqref{candi}.
There is a subsequence of $\{w_k\}$, which we denote by itself,
such that
\begin{equation}
\lim_{k\to\infty}\sup_{0\leq t\leq T}|\langle w_k(t)-w(t),\phi\rangle|=0, \qquad
\forall\phi\in L^2_\sigma(\Omega),
\label{conv-2}
\end{equation}
\begin{equation}
\lim_{k\to\infty}\|w_k-w\|_{L^2(0,T; L^2(\Omega_L))}=0, \qquad
\forall L\in [R,\infty),
\label{conv-3}
\end{equation}
\begin{equation}
\lim_{k\to\infty}\|J_kw_k-w\|_{L^2(0,T; L^2(\Omega_L))}=0, \qquad
\forall L\in [R,\infty),
\label{conv-4}
\end{equation}
where $\Omega_L=\Omega\cap B_L$
and
$R$ is as in \eqref{cov}.
Furthermore, we have
\begin{equation}
w\in C_w([0,T]; L^2_\sigma(\Omega)),
\label{weak-conti}
\end{equation}
\begin{equation}
\lim_{t\to 0}\|w(t)-w_0\|_2=0.
\label{ini}
\end{equation}
\label{str-conv}
\end{lemma}
\begin{proof}
We first fix $\phi\in C_{0,\sigma}^\infty(\Omega)$.
By \eqref{appro-est} it is obvious that
$\langle w_k,\phi\rangle$ is uniformly bounded.
Let $0\leq s<t\leq T$, then we see 
from 
\eqref{LH}, \eqref{LS},
\eqref{auxi-est-2}, \eqref{f-est}, \eqref{NS-appro} and \eqref{appro-est} that
\begin{equation*}
\begin{split}
&|\langle w_k(t),\phi\rangle- \langle w_k(s),\phi\rangle| \\
&\leq\int_s^t\Big[\|\nabla w_k\|_2\|\nabla\phi\|_2
+|h|_\infty(|u_\infty|+\|u_s\|_\infty)\|\nabla w_k\|_2\|\phi\|_2  \\
&\quad +\|\widetilde U\|_{3,\infty}\|\nabla w_k\|_2\|\phi\|_{6,2}
+|h|_\infty\|\nabla u_s\|_\infty\|w_k\|_2\|\phi\|_2
+\|\nabla\widetilde U\|_{3,\infty}\|w_k\|_2\|\phi\|_{6,2} \\
&\quad +C\|w_k\|_2^{1/2}\|\nabla w_k\|_2^{3/2}\|\phi\|_6+\|f\|_2\|\phi\|_2\Big] d\tau \\
&\leq
CY(T)^{1/2}\Big\{(\|\nabla\phi\|_2+\|\phi\|_2)(t-s)^{1/2}
+\|\phi\|_2(t-s)+\|\nabla\phi\|_2(t^{1/2}-s^{1/2})\Big\} \\
&\quad +CY(T)\|\nabla\phi\|_2(t-s)^{1/4}
+C\|\phi\|_2(t^{1/4}-s^{1/4}).
\end{split}
\end{equation*}
This shows that $\langle w_k,\phi\rangle$ is equi-continuous on $[0,T]$.
By the Ascoli-Arzel\`a theorem, $\{\langle w_k,\phi\rangle\}$ contains a subsequence
(dependent of $\phi\in C_{0,\sigma}^\infty(\Omega)$) which is
uniformly convergent on $[0,T]$.
Since $L^2_\sigma(\Omega)$ is separable, the diagonal method concludes that
one can further take a subsequence of $\{w_k\}$
(independent of $\phi\in L^2_\sigma(\Omega)$),
which is denoted by
itself, such that \eqref{conv-2} holds true.
This immediately implies \eqref{weak-conti}, and thereby
$\|w_0\|_2^2\leq\liminf_{t\to 0}\|w(t)\|_2^2$.
On the other hand, $\|w(t)\|_2^2$ is bounded from above by
the RHS of \eqref{appro-est}, which implies that
$\limsup_{t\to 0}\|w(t)\|_2^2\leq \|w_0\|_2^2$.
We thus obtain \eqref{ini}.

Let $L\in [R,\infty)$, and fix a cut-off function $\psi\in C_0^\infty(B_{2L})$
satisfying $\psi=1$ on $B_L$.
We utilize the Friedrichs inequality (\cite[p.489]{CH})
to see that, for every $\varepsilon >0$,
there are finite number of elements
$\phi_1,\cdots,\phi_m\in L^2(\Omega_{2L})$ such that
\begin{equation*}
\begin{split}
&\quad \|w_k(t)-w(t)\|_{2,\Omega_L}^2 \\
&\leq \|\psi(w_k(t)-w(t))\|_{2,\Omega_{2L}}^2 \\
&\leq\varepsilon\|\nabla[\psi(w_k(t)-w(t))]\|_{2,\Omega_{2L}}^2
+\sum_{j=1}^m\left|\langle\psi(w_k(t)-w(t)), \phi_j\rangle\right|^2.
\end{split}
\end{equation*}
Using \eqref{appro-est}, we find
\begin{equation*}
\begin{split}
&\quad \int_0^T\|w_k(t)-w(t)\|_{2,\Omega_L}^2\,dt  \\
&\leq C(1+T)Y(T)\,\varepsilon
+\sum_{j=1}^m \int_0^T\left|\langle w_k(t)-w(t), \mathbb P(\psi\phi_j)\rangle\right|^2.
\end{split}
\end{equation*}
By virtue of \eqref{conv-2} with $\mathbb P(\psi\phi_j)\in L^2_\sigma(\Omega)$
we obtain
\[
\limsup_{k\to\infty}\int_0^T\|w_k(t)-w(t)\|_{2,\Omega_L}^2\,dt
\leq C_T\varepsilon,
\]
which yields \eqref{conv-3}.
Finally, by \eqref{appro-est} we have
\begin{equation*}
\begin{split}
\int_0^T\|J_kw_k(t)-w_k(t)\|_{2,\Omega_L}^2\,dt
&\leq\int_0^T\left(\int_0^{1/k}\|\frac{d}{d\tau}e^{-\tau A}w_k(t)\|_2\,d\tau\right)^2
dt \\
&\leq\frac{C}{k}\int_0^T\|\nabla w_k(t)\|_2^2\,dt.
\end{split}
\end{equation*}
This combined with \eqref{conv-3} completes the proof of \eqref{conv-4}.
\end{proof}

We are in a position to provide a weak solution.
\begin{proposition}
Problem \eqref{NS-3} with \eqref{data} admits at least one weak solution.
\label{weak-sol}
\end{proposition}
\begin{proof}
The solution $w_k$ to \eqref{NS-appro} obtained in Lemma \ref{appro-sol}
fulfills
\begin{equation*}
\begin{split}
&\langle w_k(t),\varphi(t)\rangle
+\int_s^t
\Big[\langle\nabla w_k,\nabla\varphi\rangle
+\langle\{h(u_\infty+u_s)+\widetilde U\}\cdot\nabla w_k,\varphi\rangle \\
&\qquad -\langle (hu_s+\widetilde U)\otimes w_k,\nabla\varphi\rangle
+\langle (J_kw_k)\cdot\nabla w_k, \varphi\rangle\Big] d\tau \\
&=\langle w_k(s),\varphi(s)\rangle
+\int_s^t \Big[\langle w_k,\partial_\tau\varphi\rangle +\langle f,\varphi\rangle\Big]
d\tau
\end{split}
\end{equation*}
for all $0\leq s<t<\infty$ and 
$\varphi$ satisfying \eqref{test}.
It suffices to show \eqref{weak-form} under the additional condition
$\varphi\in L^\infty_{loc}([0,\infty); L^\infty(\Omega))$;
in fact, \eqref{weak-form} with $J_m\varphi\,(m=1,2,...)$
implies \eqref{weak-form} for general $\varphi$ of class \eqref{test}
by passing to the limit 
as $m\to\infty$.
We fix $T\in (0,\infty)$, and let $0\leq s<t\leq T$.
As in the standard Navier-Stokes theory, it follows from \eqref{conv-1} together
with Lemma \ref{str-conv} that
\begin{equation}
\lim_{k\to \infty}
\int_s^t\langle (J_kw_k)\cdot\nabla w_k, \varphi\rangle d\tau
=\int_s^t\langle w\cdot\nabla w, \varphi\rangle d\tau.
\label{nonl-conv}
\end{equation}
Indeed, for every $\varepsilon >0$, one can take
$L=L(\varepsilon,T)\in [R,\infty)$ so large,
independent of $k$ on account of \eqref{appro-est}, that
\begin{equation*}
\begin{split}
&\quad 
\left|
\int_s^t\langle (J_kw_k-w)\cdot\nabla w_k, (1-\chi_{B_L})\varphi\rangle d\tau
\right|  \\
&\leq CY(T)
\left(\int_0^T\|\varphi(\tau)\|_{6,\mathbb R^3\setminus B_L}^4\,d\tau\right)^{1/4}
\leq\varepsilon,
\end{split}
\end{equation*}
where $\chi_{B_L}$ stands for the characteristic function on $B_L$.
We then find from \eqref{conv-4} that
\begin{equation*}
\begin{split}
&\quad \limsup_{k\to\infty}
\left|\int_s^t\langle (J_kw_k-w)\cdot\nabla w_k, \varphi\rangle d\tau\right| \\
&\leq\varepsilon +\lim_{k\to\infty}
\left|\int_s^t\langle (J_kw_k-w)\cdot\nabla w_k, \chi_{B_L}\varphi\rangle d\tau\right|
=\varepsilon,
\end{split}
\end{equation*}
which yields \eqref{nonl-conv}.
Given $\varepsilon >0$, we take $\widetilde U_\varepsilon$ in
Lemma \ref{auxi-appro}.
Then we have
\begin{equation*}
\begin{split}
&\quad \left|
\int_s^t\langle\widetilde U\otimes (w_k-w), \nabla\varphi\rangle d\tau\right| \\
&\leq
\left|\int_s^t\langle\widetilde U_\varepsilon \otimes (w_k-w),
\nabla\varphi\rangle d\tau\right|
+C\varepsilon\, Y(T)^{1/2}\|\nabla\varphi\|_{L^2(0,T; L^2(\Omega))}.
\end{split}
\end{equation*}
Since
$\sum_j\widetilde U_{\varepsilon,j}(\nabla\varphi_j)\in L^1(0,T; L^2(\Omega))$
and since $\varepsilon >0$ is arbitrary,
it follows from \eqref{conv-1} that
\[
\lim_{k\to\infty}
\int_s^t\langle\widetilde U\otimes (w_k-w), \nabla\varphi\rangle d\tau =0.
\]
The convergence of the other terms is easily verified.
Thus the function $w$ obtained in \eqref{candi} satisfies \eqref{weak-form}.

It remains to show \eqref{SEI} for $s=0$.
By \eqref{EE} we have
\begin{equation*}
\begin{split}
&\frac{1}{2}\|w_k(t)\|_2^2+\int_0^t\|\nabla w_k\|_2^2d\tau \\
&=\frac{1}{2}\|w_0\|_2^2
+\int_0^t\big[\langle (hu_s+\widetilde U)\otimes w_k, \nabla w_k\rangle 
+\langle f,w_k\rangle\big] d\tau
\end{split}
\end{equation*}
for all $t\geq 0$ and it suffices to prove
\begin{equation}
\lim_{k\to\infty}\int_0^t\langle (hu_s+\widetilde U)\otimes w_k, \nabla w_k\rangle d\tau
=\int_0^t\langle (hu_s+\widetilde U)\otimes w, \nabla w\rangle d\tau.
\label{conv-5}
\end{equation}
We fix $T\in (0,\infty)$, and let $t\in (0,T)$.
We also fix $\varepsilon >0$ arbitrarily and use
the function $\widetilde U_\varepsilon$ in Lemma \ref{auxi-appro} again to obtain
\begin{equation*}
\left|\int_0^t\langle (\widetilde U-\widetilde U_\varepsilon)\otimes
(w_k-w), \nabla w_k\rangle d\tau\right|
\leq CY(T)\varepsilon.
\end{equation*}
One can choose $L=L(\varepsilon,T)\in [R,\infty)$, independent of $k$, such that
\begin{equation*}
\begin{split}
&\quad \left|\int_0^t \langle (1-\chi_{B_L})(hu_s+\widetilde U_\varepsilon)\otimes
(w_k-w), \nabla w_k\rangle d\tau\right| \\
&\leq CY(T)\left(\int_0^T\left\{\|u_s\|_{6,\mathbb R^3\setminus B_L}
+\|\widetilde U_\varepsilon(\tau)\|_{6,\mathbb R^3\setminus B_L}\right\}^4d\tau
\right)^{1/4}
\leq\varepsilon.
\end{split}
\end{equation*}
Hence, we obtain from \eqref{conv-3} that
\begin{equation}
\begin{split}
&\quad \limsup_{k\to\infty}\left|
\int_0^t\langle (hu_s+\widetilde U)\otimes (w_k-w),
\nabla w_k\rangle d\tau\right| \\
&\leq (CY(T)+1)\,\varepsilon
+\lim_{k\to\infty}\int_0^T
\|hu_s+\widetilde U_\varepsilon\|_\infty\|\chi_{B_L}(w_k-w)\|_2
\|\nabla w_k\|_2 d\tau  \\
&=(CY(T)+1)\,\varepsilon.
\end{split}
\label{conv-6}
\end{equation}
On the other hand, since
\[
\|(hu_s+\widetilde U)\otimes w\|_2
\leq C(|h|_\infty\|u_s\|_3 +\|\widetilde U\|_{3,\infty})\|\nabla w\|_2
\in L^2(0,T),
\]
we have
\[
\lim_{k\to \infty}
\int_0^t\langle (hu_s+\widetilde U)\otimes w, \nabla w_k-\nabla w\rangle d\tau =0.
\]
This together with \eqref{conv-6} concludes \eqref{conv-5}.
\end{proof}

We conclude this section with the proof of the strong
energy inequality \eqref{SEI}.
\begin{proposition}
The solution obtained in Proposition \ref{weak-sol} enjoys
\eqref{SEI} for $s=0$, a.e. $s>0$ and all $t\geq s$.
\label{SEI-hold}
\end{proposition}
\begin{proof}
The case $s=0$ has been already shown in the proof of
Proposition \ref{weak-sol}.
Let $T\in (0,\infty)$.
To consider the other case $s\in (0,T)$,
let us take a subsequence of $\{w_k\}$, which is still denoted by itself,
and a set $J\subset (0,T)$ with the Lebesgue measure
$|J|=0$ such that
\begin{equation}
\lim_{k\to\infty}\|w_k(t)-w(t)\|_{2,\Omega_L}=0, \qquad
\forall L\in [R,\infty),\, \forall t\in (0,T)\setminus J,
\label{pointwise}
\end{equation}
where
$\Omega_L=\Omega\cap B_L$ and
$R$ is as in \eqref{cov}.
This is in fact verified as follows:
For each $i=1,2,...$,
it follows from \eqref{conv-3} that one can take a subsequence of
$\{w_k\}$, denoted by itself, and a set $J_i\subset (0,T)$
with $|J_i|=0$ such that
\[
\lim_{k\to\infty}\|w_k(t)-w(t)\|_{2,\Omega_{R+i}}=0, \qquad
\forall t\in (0,T)\setminus J_i.
\]
Then, by the diagonal method, we are led to \eqref{pointwise}
for a suitable subsequence of $\{w_k\}$, where
$J=\cup_{i=1}^\infty J_i$.

Let us go back to the approximate problem \eqref{NS-appro} together with
the pressure $p_k$ associated with the strong solution $w_k$ obtained in
Lemma \ref{appro-sol}:
\begin{equation}
\begin{split}
&\partial_tw_k+(J_kw_k)\cdot\nabla w_k+Sw_k
=\Delta w_k-\nabla p_k+f, \\
&\mbox{div $w_k$}=0, \\
&w_k|_{\partial\Omega}=0, \\
&w_k\to 0 \quad\mbox{as $|x|\to\infty$}, \\
&w_k(\cdot,0)=w_0.
\end{split}
\label{appro-pde}
\end{equation}
In order to control the behavior of the pressure $p_k$
at infinity uniformly in $k$,
it is convenient to split the solution $w_k$ into three parts
\[
w_k=w_k^1+w_k^2+w_k^3, \qquad
p_k=p_k^1+p_k^2+p_k^3,
\]
where
\begin{equation}
\partial_tw_k^1-\Delta w_k^1+\nabla p_k^1=-hu_\infty\cdot\nabla w_k+f, \quad
w_k^1(\cdot,0)=w_0,
\label{appro-1}
\end{equation}
\begin{equation}
\partial_tw_k^2-\Delta w_k^2+\nabla p_k^2=-(J_kw_k)\cdot\nabla w_k, \quad
w_k^2(\cdot,0)=0,
\label{appro-2}
\end{equation}
\begin{equation}
\partial_tw_k^3-\Delta w_k^3+\nabla p_k^3
=-(hu_s+\widetilde U)\cdot\nabla w_k-w_k\cdot\nabla (hu_s+\widetilde U), \quad
w_k^3(\cdot,0)=0,
\label{appro-3}
\end{equation}
subject to
\[
\mbox{div $w_k^j$}=0, \qquad
w_k^j|_{\partial\Omega}=0, \qquad
w_k^j\to 0
\quad\mbox{as $|x|\to\infty$}
\]
for $j=1,2,3$.

Let us begin with \eqref{appro-1}.
By the standard energy method together with \eqref{appro-est},
\eqref{f-class} and the Gronwall argument, we have
\begin{equation}
\|w_k^1(t)\|_2^2+\int_0^t\|\nabla w_k^1\|_2^2d\tau\leq CY(T)e^T,
\label{app-1-1st}
\end{equation}
and
\[
\|\nabla w_k^1(t)\|_2^2+\int_s^t\|Aw_k^1\|_2^2d\tau
\leq\|\nabla w_k^1(s)\|_2^2+
2\int_s^t\|f\|_2^2d\tau +CY(T),
\]
for $0<s<t \leq T$.
Integration of the latter with respect to $s$
over $(0,t)$ together with \eqref{f-est} and \eqref{app-1-1st} yield
\[
t\|\nabla w_k^1(t)\|_2^2+\int_0^t\tau\|Aw_k^1\|_2^2d\tau\leq C_T,
\]
which implies
\begin{equation}
\int_s^t\|Aw_k^1\|_2^2d\tau\leq\frac{C_T}{s}
\label{app-1-2nd}
\end{equation}
for $0<s<t \leq T$.
In view of the equation of \eqref{appro-1} and by use of estimate
$\|\nabla^2g\|_2\leq C(\|Ag\|_2+\|\nabla g\|_2)$ for $g\in D_2(A)$
(see Heywood \cite{He2}),
we gather \eqref{f-est}, \eqref{appro-est}, \eqref{app-1-1st} and
\eqref{app-1-2nd} to find
\begin{equation*}
\sup_k\int_s^T\|\nabla p_k^1\|_2^2d\tau <\infty.
\end{equation*}
By the embedding relation, there are constants $c_k^1\,(k=1,2,...)$ such that
\[
\sup_k\int_s^T\|p_k^1+c_k^1\|_6^2d\tau <\infty.
\]
Hence, one finds a subsequence of 
$\{p_k^1\}$
(dependent of each $s\in (0,T)$),
which one denotes by itself, as well as
$p^1\in L^2(s,T; L^6(\Omega))$ with
$\nabla p^1\in L^2(s,T; L^2(\Omega))$ so that
\begin{equation}
\begin{split}
&p_k^1+c_k^1\to p^1\quad\mbox{weakly in $L^2(s,T; L^6(\Omega))$}, \\
&\nabla p_k^1\to\nabla p^1\quad\mbox{weakly in $L^2(s,T; L^2(\Omega))$},
\end{split}
\label{conv-p}
\end{equation}
as $k\to\infty$.

We next consider \eqref{appro-2},
but this part is exactly the same as in \cite{MS}.
From \eqref{appro-est} we deduce
\[
\sup_k\int_0^T\|(J_kw_k)\cdot\nabla w_k\|_{5/4}^{5/4}\, d\tau <\infty.
\]
Then the maximal regularity for the Stokes system
(see Solonnikov \cite{Sol}, Giga and Sohr \cite{GS}) leads to
\begin{equation}
\sup_k\int_0^T\|p_k^2+c_k^2\|_{15/7}^{5/4}\, d\tau
\leq C \sup_k\int_0^T\|\nabla p_k^2\|_{5/4}^{5/4}\, d\tau <\infty
\label{p2-est}
\end{equation}
for some constants $c_k^2\,(k=1,2,...)$.

We turn to \eqref{appro-3}.
We fix $q\in (1,2)$ and take $p\in (3,\infty)$ satisfying $3/2p+1/q>1$.
By \eqref{appro-est} and by
\eqref{auxi-est-1}--\eqref{auxi-est-3}
we see that
\begin{equation*}
\begin{split}
&\quad \|(hu_s+\widetilde U)\cdot\nabla w_k+w_k\cdot\nabla (hu_s+\widetilde U)\|_r \\
&\leq \big(|h|_\infty\|u_s\|_p+\|\widetilde U\|_p\big)\|\nabla w_k\|_2
+\big(|h|_\infty\|\nabla u_s\|_p+\|\nabla\widetilde U\|_p\big)\|w_k\|_2  \\ 
&\leq C\big(1+\tau^{-1/2+3/2p}\big)\|\nabla w_k\|_2 
+C\big\{1+\tau^{-1+3/2p}(1+T)^{1/2-3/2p}\big\}Y(T)^{1/2},
\end{split}
\end{equation*}
for $\tau\in (0,T)$, where $r\in (6/5,2)$ satisfies $1/r=1/p+1/2$, and therefore
\[
\sup_k\int_0^T
\|(hu_s+\widetilde U)\cdot\nabla w_k+w_k\cdot\nabla (hu_s+\widetilde U)\|_r^q\, d\tau
<\infty.
\]
By the same reasoning as above,
we obtain
\begin{equation}
\sup_k\int_0^T\|p_k^3+c_k^3\|_{r_*}^q\, d\tau
\leq C\sup_k\int_0^T\|\nabla p_k^3\|_r^q\, d\tau<\infty,
\label{p3-est}
\end{equation}
for some constants $c_k^3\,(k=1,2,...)$,
where $1/r_*=1/r-1/3$.

We now fix $s\in (0,T)\setminus J$, and let $t\in (s,T]$, where $J$ is as in
\eqref{pointwise}.
We take a cut-off function
$\psi\in C_0^\infty(B_2)$ such that $\psi=1$ on $B_1$ as well as
$\psi\geq 0$, and set
$\psi_L(x)=\psi(x/L)$ for $L\geq R$, where $R$ is as in \eqref{cov}.
We multiply the equation of \eqref{appro-pde}
by $\psi_Lw_k$ and integrate the resulting formula
over $\Omega\times (s,t)$ to find
\begin{equation}
\begin{split}
&\quad \frac{1}{2}\|\sqrt{\psi_L}w_k(t)\|_2^2
+\int_s^t\Big(\|\sqrt{\psi_L}\nabla w_k\|_2^2+\langle\nabla p_k^1, \psi_Lw_k\rangle
\Big)d\tau \\
&=\frac{1}{2}\|\sqrt{\psi_L}w_k(s)\|_2^2
+\int_s^t\Big(-\langle\nabla\psi_L\cdot\nabla w_k, w_k\rangle \\
&\quad +\langle p_k^2+c_k^2, w_k\cdot\nabla\psi_L\rangle 
+ \langle p_k^3+c_k^3, w_k\cdot\nabla\psi_L\rangle  \\
&\quad +\left\langle\frac{|w_k|^2}{2},
\{J_kw_k+h(u_\infty+u_s)+\widetilde U\}\cdot\nabla\psi_L\right\rangle \\
&\quad +\langle (hu_s+\widetilde U)\cdot w_k, w_k\cdot\nabla\psi_L\rangle
+\langle (hu_s+\widetilde U)\otimes w_k, (\nabla w_k)\psi_L\rangle \\
&\quad +\langle f, \psi_Lw_k\rangle \Big) d\tau.
\end{split}
\label{energy-cut}
\end{equation}
On account of \eqref{p2-est} and \eqref{p3-est}, we
observe
\begin{equation}
\begin{split}
&\quad \left|\int_s^t
\langle p_k^2+c_k^2, w_k\cdot\nabla\psi_L\rangle
+\langle p_k^3+c_k^3, w_k\cdot\nabla\psi_L\rangle d\tau\right| \\
&\leq C_T\|w_k\cdot\nabla\psi_L\|_{L^5(0,T;L^{15/8}(\Omega))}
+C_T\|w_k\cdot\nabla\psi_L\|_{L^{q^\prime}(0,T;L^{(r_*)^\prime}(\Omega))} \\
&\leq C_T\big(\|\nabla\psi_L\|_{30}+\|\nabla\psi_L\|_\sigma\big),
\end{split}
\label{p-est}
\end{equation}
where $1/q^\prime+1/q=1$,
$1/(r_*)^\prime+1/r_*=1$ and $1/\sigma=5/6-1/r$.
Note that $\sigma\in (3,\infty)$.
Making use of \eqref{LH}, \eqref{LS}, \eqref{auxi-est-2} and
\eqref{appro-est}, we find
\begin{equation}
\begin{split}
&\quad \Big|\int_s^t\Big(-\langle\nabla\psi_L\cdot\nabla w_k, w_k\rangle
+\left\langle\frac{|w_k|^2}{2}, 
\{J_kw_k+h(u_\infty+u_s)+\widetilde U\}\cdot\nabla\psi_L\right\rangle \\
&\qquad +\langle (hu_s+\widetilde U)\cdot w_k, w_k\cdot\nabla\psi_L\rangle\Big)
d\tau\Big| \\
&\leq C\int_s^t\Big[\left(1+|h|_\infty\|u_s\|_3
+\|\widetilde U\|_{3,\infty}\right)\|w_k\|_2\|\nabla w_k\|_2
+|h|_\infty |u_\infty|\|w_k\|_2^2  \\
&\qquad +\|w_k\|_2^{3/2}\|\nabla w_k\|_2^{3/2}\Big] d\tau\,\|\nabla\psi_L\|_\infty \\
&\leq C\Big\{Y(T)(T^{1/2}+T)+Y(T)^{3/2}T^{1/4}\Big\}\|\nabla\psi_L\|_\infty,
\end{split}
\label{others-est}
\end{equation}
from which together with \eqref{p-est}, we see that \eqref{energy-cut} yields
\begin{equation}
\begin{split}
&\quad \frac{1}{2}\|\sqrt{\psi_L}w_k(t)\|_2^2 
+\int_s^t\Big(\|\sqrt{\psi_L}\nabla w_k\|_2^2+\langle\nabla p_k^1, \psi_Lw_k\rangle 
\Big)d\tau \\ 
&\leq\frac{1}{2}\|\sqrt{\psi_L}w_k(s)\|_2^2 
+\int_s^t\Big(\langle (hu_s+\widetilde U)\otimes w_k, (\nabla w_k)\psi_L\rangle
+\langle f, \psi_Lw_k\rangle \Big) d\tau \\
&\quad +C_T(\|\nabla\psi_L\|_{30}
+\|\nabla\psi_L\|_\sigma
+\|\nabla\psi_L\|_\infty).
\end{split}
\label{energy-cut2}
\end{equation}

We now let $k\to\infty$ along the subsequence above.
Since $s\in (0,T)\setminus J$, we know by \eqref{pointwise} that
$\lim_{k\to\infty}\|\sqrt{\psi_L}w_k(s)\|_2=\|\sqrt{\psi_L}w(s)\|_2$.
We split
\[
\int_s^t
\Big(\langle (hu_s+\widetilde U)\otimes w_k, (\nabla w_k)\psi_L\rangle
-\langle (hu_s+\widetilde U)\otimes w, (\nabla w)\psi_L\rangle\Big) d\tau
\]
into two parts $I+II$, where
\begin{equation*}
\begin{split}
|I|
&=\left|\int_s^t\langle (hu_s+\widetilde U)\otimes (w_k-w), 
(\nabla w_k)\psi_L\rangle\right| \\
&\leq CY(T)^{1/2}
\left(|h|_\infty\|u_s\|_\infty+\sup_{s\leq\tau\leq t}\|\widetilde U(\tau)\|_\infty\right)
\|w_k-w\|_{L^2(0,T; L^2(\Omega_{2L}))},
\end{split}
\end{equation*}
while
\[
|II|
=\left|\int_s^t\langle (hu_s+\widetilde U)\otimes w, (\nabla w_k-\nabla w)\psi_L
\rangle\right|\to 0\quad (k\to\infty)
\]
is easily verified by \eqref{conv-1}.
Since $\|\widetilde U(\tau)\|_\infty\leq Cs^{-1/2}$ for $\tau\geq s>0$
by \eqref{auxi-est-1},
Lemma \ref{str-conv} implies that
$\lim_{k\to\infty}I=0$, too.
From
\eqref{conv-1}, \eqref{conv-2}, \eqref{conv-3} and
\eqref{conv-p} as well as the observation above
we deduce that \eqref{energy-cut2} leads to
\begin{equation}
\begin{split}
&\quad \frac{1}{2}\|\sqrt{\psi_L}w(t)\|_2^2
+\int_s^t\Big(\|\sqrt{\psi_L}\nabla w\|_2^2+\langle\nabla p^1, \psi_Lw\rangle
\Big)d\tau \\
&\leq\frac{1}{2}\|\sqrt{\psi_L}w(s)\|_2^2
+\int_s^t\Big(\langle (hu_s+\widetilde U)\otimes w, (\nabla w)\psi_L\rangle
+\langle f, \psi_Lw\rangle \Big) d\tau \\
&\quad +C_T(\|\nabla\psi_L\|_{30}
+\|\nabla\psi_L\|_\sigma
+\|\nabla\psi_L\|_\infty).
\end{split}
\label{energy-cut3}
\end{equation}
Here, we have
\begin{equation*}
\begin{split}
\left|\int_s^t\langle \nabla p^1, \psi_Lw\rangle d\tau\right|
&=\left|-\int_s^t\langle p^1, w\cdot\nabla\psi_L\rangle d\tau\right| \\
&\leq \|\nabla\psi_L\|_3\int_s^t\|w\|_{L^2(A_L)}\|p^1\|_{L^6(A_L)}d\tau,
\end{split}
\end{equation*}
where $A_L=B_{2L}\setminus\overline{B_L}$.
By the Lebesgue convergence theorem, we see that
$\int_s^t\cdots\to 0$ as $L\to\infty$.
Therefore, by passing to the limit 
as
$L\to\infty$ in \eqref{energy-cut3},
we arrive at \eqref{SEI} for all $s\in (0,T)\setminus J$ and $t\in (s,T]$.
\end{proof}

\section{Strong solution}\label{strong-iden}

Let $\bar t\in (T_0,\infty)$, where $T_0$ is as in \eqref{h-2}
(resp. \eqref{h-3}) for the starting (resp. landing) problem.
In this section we construct a strong solution to \eqref{NS-3} 
with \eqref{data} on the interval $[\bar t,\infty)$
under a certain smallness condition on $w(\bar t)$.
And then, it is identified on $[\bar t,\infty)$ with the weak solution
obtained in the previous section.

The first two propositions in this section are
independent of the argument in the previous section.
By \eqref{h-2} problem \eqref{NS-3} on $[\bar t,\infty)$
is formally converted into the integral equation
\begin{equation}
w=\Psi w \qquad (t\geq \bar t)
\label{int-bar}
\end{equation}
with
\begin{equation*}
\begin{split}
(\Psi w)(t)&=H(t)-\int_{\bar t}^tT(t-\tau)\mathbb P[(u_s+\widetilde U)\cdot\nabla w \\
&\qquad\qquad  +w\cdot\nabla (u_s+\widetilde U)+w\cdot\nabla w](\tau) d\tau, \\
H(t)&=T(t-\bar t)w(\bar t)+H_f(t), \\
H_f(t)&=\int_{\bar t}^t T(t-\tau)\mathbb P f(\tau) d\tau,
\end{split}
\end{equation*}
where the term 
$u_s\cdot\nabla w+w\cdot\nabla u_s$
is absent for the landing problem and
\[
T(t)=
\left\{
\begin{array}{ll}
e^{-tA_{u_\infty}} \qquad & (\mbox{starting problem}),  \\
e^{-tA} & (\mbox{landing problem}).
\end{array}
\right.
\]
We take a small $w(\bar t)$ from $L^3_\sigma(\Omega)$ and look for
a solution in a closed ball
\begin{equation}
E_\rho=\{w\in E; \|w\|_E\leq\rho\}
\label{E-ball}
\end{equation}
of the Banach space
\begin{equation}
\begin{split}
E=&\{w\in C((\bar t,\infty); L^6_\sigma(\Omega)\cap L^\infty(\Omega));\,
\nabla w\in C((\bar t,\infty); L^3(\Omega)),  \\
&\qquad\qquad\qquad 
\|w\|_E:=\sup_{t\in (\bar t,\infty)}\phi_w(t)<\infty,\,\lim_{t\to\bar t+0}\phi_w(t)=0\}
\end{split}
\label{E}
\end{equation}
endowed with norm 
$\|\cdot\|_E$, where
\[
\phi_w(t)
:=(t-\bar t)^{1/2}\big(\|w(t)\|_\infty+\|\nabla w(t)\|_3\big)+(t-\bar t)^{1/4}\|w(t)\|_6.
\]
Since we need the smallness of $|u_\infty|$ to get a unique
steady flow $u_s$ for the starting problem, see \eqref{steady},
we may assume at the beginning that $|u_\infty|\leq\delta_0$.
This is not needed for the landing problem.

Let us start with the following lemma on $H_f(t)$.
\begin{lemma}
Let
\[
\frac{4}{3}<q<\frac{3}{2}<r<3.
\]
Then we have $H_f\in E$ and
\begin{equation}
\|H_f\|_E\leq
\gamma (1+\|\nabla u_s\|_q+\|\nabla u_s\|_r)
(\|U(\bar t)\|_{3,\infty,\mathbb R^3}+\|U(\bar t)\|_{3,\infty,\mathbb R^3}^2),
\label{Hf-est}
\end{equation}
with some constant $\gamma=\gamma(q,r)>0$. For the landing problem,
the term $\|\nabla u_s\|_q+\|\nabla u_s\|_r$ is absent.
\label{force-est}
\end{lemma}
\begin{proof}
We derive only 
\eqref{Hf-est}
since the continuity 
in $t$ (as in \eqref{Phi}) and $\lim_{t\to\bar t+0}\phi_{H_f}(t)=0$ 
are easily verified so that $H_f\in E$ 
(the latter follows from the fact that $f(t)$ does not possess 
any singular behavior near $t=\bar t$).
We divide the external force, see \eqref{data}, into two parts:
\[
f=f_0-\widetilde U\cdot\nabla\widetilde U, \quad
f_0=F-(u_s\cdot\nabla\widetilde U+\widetilde U\cdot\nabla u_s) \quad (t\geq \bar t).
\]
By \eqref{auxi-bar} we obtain
\[
\|u_s\cdot\nabla\widetilde U+\widetilde U\cdot\nabla u_s\|_p
\leq C(t-\bar t)^{-1/2}
\|\nabla u_s\|_p
\|U(\bar t)\|_{3,\infty,\mathbb R^3}
\]
for all $t>\bar t$ and $p\in (4/3,3)$, which combined with \eqref{F-bar}
leads to
\begin{equation}
\|f_0(t)\|_p\leq C
m_p
(t-\bar t)^{-1/2}\|U(\bar t)\|_{3,\infty,\mathbb R^3}
\label{f-0}
\end{equation}
for the same $p$ as above,
where we put $m_p=1+\|\nabla u_s\|_p$ for notational simplicity 
($m_p=1$ for the landing problem).
We fix $q$ and $r$ such that
\[
\frac{4}{3}<q<\frac{3}{2}<r<3
\]
and split $H_{f_0}(t)$ into
\[
H_{f_0}(t)=\left(\int_{\bar t}^{(\bar t+t)/2}+\int_{(\bar t+t)/2}^t\right)
T(t-\tau)\mathbb Pf_0(\tau) d\tau
=:H_{f_0,1}(t)+H_{f_0,2}(t).
\]
We are going to employ \eqref{d-1} and \eqref{d-2}.
From \eqref{f-0} we deduce
\begin{equation*}
\begin{split}
\|H_{f_0,1}(t)\|_\infty+\|\nabla H_{f_0,1}(t)\|_3
&\leq C\int_{\bar t}^{(\bar t+t)/2}
(t-\tau)^{-1}\|f_0(\tau)\|_{3/2}d\tau \\
&\leq C m_{3/2}
(t-\bar t)^{-1/2}\|U(\bar t)\|_{3,\infty,\mathbb R^3} \qquad (t>\bar t)
\end{split}
\end{equation*}
and
\begin{equation*} 
\begin{split}
&\quad \|H_{f_0,2}(t)\|_\infty+\|\nabla H_{f_0,2}(t)\|_3 \\
&\leq C\int_{(\bar t+t)/2}^t
(t-\tau)^{-3/2r}\|f_0(\tau)\|_r d\tau \\
&\leq C m_r
(t-\bar t)^{1/2-3/2r}\|U(\bar t)\|_{3,\infty,\mathbb R^3} \qquad
(\bar t<t\leq \bar t+2).
\end{split}
\end{equation*}
To estimate $H_{f_0,2}(t)$ for $t>\bar t+2$, we further split it into
\[
H_{f_0,2}(t)=\int_{(\bar t+t)/2}^{t-1}+\int_{t-1}^t
=:H_{f_0,21}(t)+H_{f_0,22}(t).
\]
Then we find
\begin{equation*}
\begin{split}
\|H_{f_0,21}(t)\|_\infty+\|\nabla H_{f_0,21}(t)\|_3
&\leq C\int_{(\bar t+t)/2}^{t-1}
(t-\tau)^{-3/2q}\|f_0(\tau)\|_q d\tau \\
&\leq C m_q
(t-\bar t)^{-1/2}\|U(\bar t)\|_{3,\infty,\mathbb R^3} \qquad
(t>\bar t+2),
\end{split}
\end{equation*}
and
\begin{equation*}
\begin{split} 
\|H_{f_0,22}(t)\|_\infty+\|\nabla H_{f_0,22}(t)\|_3
&\leq C\int_{t-1}^t
(t-\tau)^{-3/2r}\|f_0(\tau)\|_r d\tau \\
&\leq C m_r
(t-\bar t)^{-1/2}\|U(\bar t)\|_{3,\infty,\mathbb R^3} \qquad
(t>\bar t+2).
\end{split}
\end{equation*}
It is easy to estimate $\|H_{f_0}(t)\|_6$ without any splitting
by use of \eqref{f-0} for $p=3/2$.
The other term $\widetilde U\cdot\nabla \widetilde U$ should be treated
separately because it does not belong to $L^q(\Omega)$ with $q\leq 3/2$;
however, the treatment is easier without any splitting on account of the faster
decay
\[
\|\widetilde U\cdot\nabla \widetilde U\|_2
\leq C(t-\bar t)^{-3/4}\|U(\bar t)\|_{3,\infty,\mathbb R^3}^2 \qquad
(t>\bar t)
\]
which follows from 
\eqref{interpo-1},
\eqref{LH} and \eqref{auxi-bar}.
The proof is complete.
\end{proof}

The following proposition provides a solution to \eqref{int-bar} with some decay
properties.
Indeed we know by \eqref{U-decay} that \eqref{small-3}
below is accomplished for large $\bar t$,
but this will be taken into consideration together with the other
smallness condition \eqref{small-2} in the proof of the main 
theorems.
\begin{proposition}
Let
\[
\frac{4}{3}<q<\frac{3}{2}<r<3.
\]
There are constants
$\delta_j=\delta_j(q,r)>0\; (j=1,3)$
and $\delta_2>0$ (independent of $q,\,r$)
such that if
\begin{equation}
|u_\infty|\leq\delta_0, \qquad
\|\nabla u_s\|_q+\|\nabla u_s\|_r \leq\delta_1,
\label{small-1}
\end{equation}
\begin{equation}
w(\bar t)\in L^3_\sigma(\Omega), \qquad
\|w(\bar t)\|_3\leq \delta_2,
\label{small-2}
\end{equation}
\begin{equation}
\|U(\bar t)\|_{3,\infty,\mathbb R^3}\leq
\delta_3,
\label{small-3}
\end{equation}
where $\delta_0$ is as in \eqref{steady},
then equation \eqref{int-bar} admits a unique solution
\begin{equation}
w\in E\cap C([\bar t,\infty); L^3_\sigma(\Omega)),
\label{cl-basic}
\end{equation}
see \eqref{E}, subject to
\[
\lim_{t\to \bar t+0}\|w(t)-w(\bar t)\|_3=0, \qquad
\|w(t)\|_3\leq C\|w(\bar t)\|_3 \quad (t\geq \bar t).
\]
For the landing problem, the condition \eqref{small-1} is redundant.
\label{mild}
\end{proposition}
\begin{proof}
We follow the method of Kato \cite{Ka} by use of \eqref{d-1} and \eqref{d-2}.
Let $w\in E$.
Then the continuity of $\Psi w$ (as in \eqref{Phi}) and 
$\lim_{t\to\bar t+0}\phi_{\Psi w}(t)=0$ 
as the properties of elements of $E$ are easily verified.
By using
\[
\nabla u_s\in L^q(\Omega)\cap L^r(\Omega), \qquad
u_s\in L^{q_*}(\Omega)\cap L^{r_*}(\Omega),
\]
where $1/q_*=1/q-1/3$ and $1/r_*=1/r-1/3$, and by splitting
the integral over $(\bar t,t)$ in the same way as in the proof of
Lemma \ref{force-est} (see also Chen \cite{C}, Enomoto and Shibata \cite{ES2}),
the term
$u_s\cdot\nabla w+w\cdot\nabla u_s$
can be treated.
From this together with
\eqref{Hf-est}
(in which
$\|U(\bar t)\|_{3,\infty,\mathbb R^3}^2$ is replaced just by
$\|U(\bar t)\|_{3,\infty,\mathbb R^3}$ if assuming that it is less than one, see
\eqref{small-3}
with \eqref{how-small} below)
and \eqref{auxi-bar} we deduce
\begin{equation*}
\begin{split}
\|\Psi w\|_E
&\leq c_1\|w(\bar t)\|_3+
2\gamma (1+\|\nabla u_s\|_q+\|\nabla u_s\|_r)
\|U(\bar t)\|_{3,\infty,\mathbb R^3}  \\
&\quad +c_2\big(\|\nabla u_s\|_q+\|\nabla u_s\|_r
+\|U(\bar t)\|_{3,\infty,\mathbb R^3}\big)
\|w\|_E +c_3\|w\|_E^2,
\end{split}
\end{equation*}
where the only term one uses the Lorentz norm is $w\cdot\nabla \widetilde U$,
that is,
\[
\|w\cdot\nabla \widetilde U\|_{2,6}\leq\|w\|_6\|\nabla\widetilde U\|_{3,\infty},
\]
see \eqref{LH},
which is combined with $L^{2,6}$-$L^r$ estimate ($r=3,6,\infty$)
of the semigroup; indeed, such estimate is a simple consequence
of \eqref{d-1}
and
\eqref{d-2} by real interpolation.
Similarly, we have
\begin{equation*}
\begin{split}
\|\Psi w_1-\Psi w_2\|_E
&\leq c_2
\big(\|\nabla u_s\|_q+\|\nabla u_s\|_r+\|U(\bar t)\|_{3,\infty,\mathbb R^3}\big)
\|w_1-w_2\|_E  \\
&\quad +c_3(\|w_1\|_E+\|w_2\|_E)\|w_1-w_2\|_E
\end{split}
\end{equation*}
for $w_1, w_2\in E$, where $c_2$ and $c_3$ are the same constants as above.
Let us take
\[
\rho=2\Big\{
c_1\|w(\bar t)\|_3+
2\gamma (1+\|\nabla u_s\|_q+\|\nabla u_s\|_r)
\|U(\bar t)\|_{3,\infty,\mathbb R^3}\Big\}
\]
and $w\in E_\rho$, see \eqref{E-ball}.
We set
\begin{equation}
\delta_1=\frac{1}{8c_2}, \qquad
\delta_2=\frac{1}{16c_1c_3}, \qquad
\delta_3=\min\Big\{\delta_1,\;
\frac{1}{32\gamma (1+\delta_1)c_3}
\Big\}.
\label{how-small}
\end{equation}
Then the conditions \eqref{small-1}, \eqref{small-2} and \eqref{small-3}
imply $\rho\leq 1/4c_3$, so that
\begin{equation*}
\begin{split}
&\|\Psi w\|_E\leq \rho \qquad\mbox{for $w\in E_\rho$}, \\
&\|\Psi w_1-\Psi w_2\|_E\leq\frac{3}{4}\|w_1-w_2\|_E \qquad
\mbox{for $w_1, w_2\in E_\rho$}.
\end{split}
\end{equation*}
We thus obtain a unique solution $w\in E_\rho$ to \eqref{int-bar}.
The proof of additional properties of $w(t)$ in the statement is standard and may
be omitted.
\end{proof}

Indeed the solution obtained in Proposition \ref{mild}
is a strong solution with values in $L^3_\sigma(\Omega)$,
but we need the following $L^2$-strong solution for later use rather than the
$L^3$-strong solution.
\begin{proposition}
Let $w(t)$ be the solution to \eqref{int-bar} obtained in Proposition \ref{mild}.
We further assume that
$w(\bar t)\in L^2_\sigma(\Omega)$.

\begin{enumerate}
\item
The solution is of class
\begin{equation}
w\in C([\bar t,\infty); L^2_\sigma(\Omega))
\cap C((\bar t,\infty); D_2(A))
\cap C^1((\bar t,\infty); L^2_\sigma(\Omega))
\label{cl-L2}
\end{equation}
subject to
$\lim_{t\to \bar t+0}\|w(t)-w(\bar t)\|_2=0$.
It also satisfies the equation
\begin{equation}
\partial_tw+Aw+
\mathbb P[(u_\infty+u_s+\widetilde U)\cdot\nabla w+w\cdot\nabla (u_s+\widetilde U)
+w\cdot\nabla w]=\mathbb Pf
\label{NS-evo}
\end{equation}
in $L^2_\sigma(\Omega)$ and the energy equality
\begin{equation}
\begin{split}
&\quad \frac{1}{2}\|w(t)\|_2^2+\int_{\bar t}^t\|\nabla w\|_2^2 d\tau  \\
&=\frac{1}{2}\|w(\bar t)\|_2^2
+\int_{\bar t}^t \Big[\langle (u_s+\widetilde U)\otimes w, \nabla w\rangle
+\langle f,w\rangle\Big] d\tau
\end{split}
\label{EE-int}
\end{equation}
for all $t\geq\bar t$ as well as
\begin{equation}
\nabla w\in L^2_{loc}([\bar t,\infty); L^2(\Omega)).
\label{1st-dissip}
\end{equation}
For the landing problem, the steady flow $u_s$ is absent in \eqref{NS-evo}
and \eqref{EE-int}.

\item
If, in addition, $w(\bar t)\in H^1_{0,\sigma}(\Omega)$, then we have
\begin{equation}
\begin{split}
&w\in L^2_{loc}([\bar t,\infty); L^\infty(\Omega)), \qquad
\nabla w\in L^\infty_{loc}([\bar t,\infty); L^2(\Omega)), \\
&\partial_tw,\, Aw\in L^2_{loc}([\bar t,\infty); L^2_\sigma(\Omega)).
\end{split}
\label{cl-ad}
\end{equation}

\end{enumerate}
\label{strong}
\end{proposition}
\begin{proof}  
Concerning the first assertion, it suffices to show that
\[
\mathbb P[f-(u_s+\widetilde U)\cdot\nabla w-w\cdot\nabla (u_s+\widetilde U)
-w\cdot\nabla w](t)
\]
is locally H\"older continuous in $t$ on the interval $(\bar t,\infty)$ 
with values in $L^2_\sigma(\Omega)$ as well as 
summable near $t=\bar t$ with values there.
The latter is obvious, for
$\|f(t)\|_2$, 
$\|\widetilde U(t)\|_6$
and $\|\nabla \widetilde U(t)\|_{3,\infty}$
do not possess any singular behavior near $t=\bar t$,
see \eqref{auxi-est-1}, \eqref{auxi-est-2} and \eqref{f-est}.
It is easy to verify the H\"older continuity locally on $(\bar t,\infty)$
of $w(t)$ with values in $L^q_\sigma(\Omega)$ for $q\geq 3$
and that of $\nabla w(t)$ with values in $L^3(\Omega)$.
This together with \eqref{auxi-reg} and \eqref{f-hoelder} lead to
the desired result.

For deduction of the second assertion,
we use the standard energy method for \eqref{NS-evo}
combined with
\[
\sup_{\bar t\leq t\leq T}\|\nabla w(t)\|_2\leq c_T, \qquad
\forall T\,\in (\bar t,\infty),
\]
which follows from estimates
of the integral equation \eqref{int-bar} 
together with \eqref{HK}
by use of $w(\bar t)\in H^1_{0,\sigma}(\Omega)$, to find
\begin{equation*}
\begin{split}
&\quad \frac{d}{dt}\|\nabla w(t)\|_2^2+\|Aw(t)\|_2^2 \\
&\leq C\Big(
\|u_\infty+u_s\|_\infty^2+\|\widetilde U(t)\|_\infty^2
+\|\nabla u_s\|_3^2
+\|\nabla\widetilde U(t)\|_{3,\infty}^2+c_T^2+c_T^4 \Big)\|\nabla w(t)\|_2^2  \\
&\quad +
C\|f(t)\|_2^2
\end{split}
\end{equation*}
for all $t\in (\bar t,T]$, where $T\in (\bar t,\infty)$ is fixed.
Note that the coefficient of $\|\nabla w\|_2^2$ 
as well as $\|f\|_2^2$
in the RHS above
belongs to $L^\infty(\bar t,T)$.
We thus employ \eqref{1st-dissip} to see that
\[
\nabla w\in L^\infty(\bar t,T; L^2(\Omega)), \qquad
Aw\in L^2(\bar t,T; L^2_\sigma(\Omega)).
\]
By the equation \eqref{NS-evo} and by
\[
\|w\|_\infty^2\leq C\|Aw\|_2\|\nabla w\|_2+
C\|\nabla w\|_2^2
\]
(see Heywood \cite{He2}), we conclude the others in \eqref{cl-ad} as well.
\end{proof}

The following proposition plays an important role
in the proof of the main theorems.
For the weak solution constructed in the previous section,
the existence of $\bar t$ satisfying the requirement below will be
shown in the following section.
\begin{proposition}
Let $\bar t\in (T_0,\infty)$, where $T_0$ is as in \eqref{h-2} or \eqref{h-3}.
Let $w(t)$ be a weak solution to \eqref{NS-3} on $[\bar t,\infty)$
with \eqref{SEI} for $s=\bar t$, and $w(\bar t)$ satisfy \eqref{small-2}
as well as $w(\bar t)\in H^1_{0,\sigma}(\Omega)$.
Assume further \eqref{small-1} and \eqref{small-3}.
By $\widetilde w(t)$ we denote the strong solution
on $[\bar t,\infty)$ to \eqref{NS-3} with initial condition
$\widetilde w(\bar t)=w(\bar t)$,
which is obtained in Proposition \ref{strong}.
Then we have
\[
w(t)=\widetilde w(t)\qquad\mbox{on $[\bar t,\infty)$},
\]
and thereby
\[
\|w(t)\|_\infty=O(t^{-1/2})\qquad\mbox{as $t\to\infty$}.
\]
For the landing problem, the condition \eqref{small-1} is redundant.
\label{identify}
\end{proposition}
\begin{proof}
We follow the argument of Serrin \cite{Se}.
In view of \eqref{cl-basic}, \eqref{cl-L2}, \eqref{1st-dissip}
and \eqref{cl-ad} one can take the strong solution $\widetilde w(t)$
as a test function, see \eqref{test}, in the relation \eqref{weak-form}
(with $s=\bar t$) for the weak solution $w(t)$.
We gather the resulting formula, \eqref{NS-evo},
\eqref{EE-int} for $\widetilde w(t)$ and 
\eqref{SEI} (with $s=\bar t$) for $w(t)$ to find
\begin{equation*}
\begin{split}
&\quad \frac{1}{2}\|w(t)-\widetilde w(t)\|_2^2
+\int_{\bar t}^t\|\nabla w-\nabla\widetilde w\|_2^2 d\tau  \\
&\leq\int_{\bar t}^t
\langle (\widetilde w+\widetilde U+u_s)\otimes (w-\widetilde w),
\nabla w-\nabla \widetilde w\rangle d\tau
\end{split}
\end{equation*}
for all $t\geq\bar t$.
By \eqref{cl-ad} together with \eqref{auxi-est-1} we know
\[
\widetilde w+\widetilde U+u_s\in L^2_{loc}([\bar t,\infty); L^\infty(\Omega)).
\]
Hence, we deduce from the inequality
\[
\|w(t)-\widetilde w(t)\|_2^2
\leq\int_{\bar t}^t \|\widetilde w+\widetilde U+u_s\|_\infty^2
\|w-\widetilde w\|_2^2 d\tau
\]
that both solutions must coincide for $t\geq \bar t$.
Thus, the large time behavior of the weak solution $w(t)$
follows from \eqref{cl-basic}.
\end{proof}

\section{Proof of main theorems}\label{main}

We are now in a position to prove the main theorems.
Let $w(t)$ be the weak solution to \eqref{NS-3} with \eqref{data}
obtained in Proposition \ref{weak-sol}.
Let us start with the energy inequality \eqref{SEI} for $s=0$.
By \eqref{F-est-3} we have
\[
|\langle f,w\rangle|
\leq\left\{C(1+t)^{-1/2}
+\|\widetilde U\otimes \widetilde U
+h(\widetilde U\otimes u_s+u_s\otimes \widetilde U)\|_2\right\}\|\nabla w\|_2.
\]
Given small $\varepsilon >0$, to be determined later, see \eqref{epsi},
we deduce from \eqref{tensor-est} that there is $T_\varepsilon >0$ such that
\[
\|\widetilde U\otimes \widetilde U
+h(\widetilde U\otimes u_s+u_s\otimes \widetilde U)\|_2^2
\leq\varepsilon t^{-1/2}, \qquad \forall t\geq T_\varepsilon,
\]
which implies that
\begin{equation}
\begin{split}
&\quad \left|\int_0^t\langle f,w\rangle d\tau\right|  \\
&\leq\frac{1}{4}\int_0^t\|\nabla w\|_2^2 d\tau+C\int_0^t\frac{d\tau}{1+\tau}
+C\int_0^{T_\varepsilon}\tau^{-1/2}d\tau +
2\varepsilon\int_{T_\varepsilon}^t\tau^{-1/2}d\tau
\end{split}
\label{EI-force}
\end{equation}
for all $t>T_\varepsilon$.
As for the second term of the RHS of \eqref{SEI}, we observe
\[
\left|\int_0^t\langle (hu_s+\widetilde U)\otimes w, \nabla w\rangle d\tau\right|
\leq c_0\int_0^t
\left(|h(\tau)|\|u_s\|_3 + \|\widetilde U(\tau)\|_{3,\infty}\right)
\|\nabla w\|_2^2 d\tau.
\]
Thanks to \eqref{auxi-decay}, there is $T_1\in (T_0,\infty)$ such that
\[
\|\widetilde U(t)\|_{3,\infty}\leq\frac{1}{8c_0}\qquad \forall t\geq T_1,
\]
where $T_0$ is as in \eqref{h-2} or \eqref{h-3}.
Suppose that the steady flow $u_s$ is so small that
\begin{equation}
\|u_s\|_3\leq\frac{1}{8c_0}.
\label{small-4}
\end{equation}
Then we get
\begin{equation}
\begin{split}
&\quad \left|\int_0^t \langle 
(hu_s+\widetilde U)\otimes w, \nabla w\rangle d\tau\right|  \\
&\leq \frac{1}{4}\int_{T_1}^t\|\nabla w\|_2^2 d\tau
+C\left(|h|_\infty\|u_s\|_3+\|v_0\|_{3,\infty}+M_3\right)
\int_0^{T_1}\|\nabla w\|_2^2 d\tau
\end{split}
\label{EI-add}
\end{equation}
for all $t> T_1$.
From \eqref{SEI} for $s=0$ together with \eqref{EI-force} and \eqref{EI-add}
we find
\begin{equation}
\|w(t)\|_2^2+\int_0^t\|\nabla w\|_2^2 d\tau
\leq K_\varepsilon+C\log (1+t)+
8\varepsilon\sqrt{t}
\label{growth-1}
\end{equation}
for all $t>\max\{T_\varepsilon, T_1\}$, and, therefore,
\begin{equation}
\int_t^{2t}\|\nabla w\|_2^2 d\tau
\leq K_\varepsilon+C\log (1+2t)+
8\varepsilon\sqrt{2t}
\label{growth-2}
\end{equation}
for all $t>\max\{T_\varepsilon/2, T_1/2\}$, where
\[
K_\varepsilon=\|w_0\|_2^2
+C\left(|h|_\infty\|u_s\|_3 +\|v_0\|_{3,\infty}+M_3\right)
\int_0^{T_1}\|\nabla w\|_2^2 d\tau
+C\sqrt{T_\varepsilon}.
\]

Let us recall the condition \eqref{small-3} in the previous section.
By \eqref{U-decay} there is
\begin{equation}
T_2 >\max\{T_\varepsilon, T_1\}
\label{T-2}
\end{equation}
such that
\begin{equation}
\|U(t)\|_{3,\infty,\mathbb R^3}\leq
\delta_3,
\qquad \forall t\geq T_2,
\label{small-33}
\end{equation}
where 
$\delta_3$ is
the constant in Propositon \ref{mild}.
By Proposition \ref{SEI-hold} we know that there is a set $J\subset (0,\infty)$
with the Lebesgue measure $|J|=0$ such that
$w(t)$ satisfies \eqref{SEI} for all $s\in (0,\infty)\setminus J$ and $t>s$.
On account of \eqref{growth-2} as well as \eqref{growth-1},
for every $t> T_2$, one can find
$\bar t\in (t,2t)\setminus J$ such that
\[
\|\nabla w(\bar t)\|_2^2
\leq\frac{2}{t}\left(K_\varepsilon+C\log (1+2t)+
8\varepsilon\sqrt{2t}\right),
\]
\[
\|w(\bar t)\|_2^2
\leq K_\varepsilon+C\log (1+2t)+
8\varepsilon\sqrt{2t},
\]
which yield
\[
\|w(\bar t)\|_3^4
\leq C\|\nabla w(\bar t)\|_2^2\|w(\bar t)\|_2^2
\leq\frac{c_*}{t}
\left[\left\{K_\varepsilon+\log (1+2t)\right\}^2+\varepsilon^2 t\right].
\]
Let $\delta_2>0$ be the constant in Proposition \ref{mild}.
We first choose and fix $\varepsilon >0$ such that
\begin{equation}
c_*\varepsilon^2\leq\frac{\delta_2^4}{2}.
\label{epsi}
\end{equation}
For such $\varepsilon >0$, we take
$T_2$ satisfying 
\eqref{T-2}--\eqref{small-33}
and then find $T_3\in (T_2,\infty)$ so that
\[
\frac{c_*}{t}\left\{K_\varepsilon+\log (1+2t)\right\}^2
\leq \frac{\delta_2^4}{2} \qquad\forall t\geq T_3.
\]
Let us fix $t\geq T_3\,(>T_2)$, for which we find 
$\bar t\in (t,2t)\setminus J$
such that $w(\bar t)\in H^1_{0,\sigma}(\Omega)$ with
\begin{equation}
\|w(\bar t)\|_3\leq\delta_2.
\label{small-22}
\end{equation}

Suppose that the steady flow $u_s$ is so amall that
\eqref{small-1} as well as \eqref{small-4} holds.
By \eqref{steady} there is a constant $\delta\in (0,\delta_0]$
such that the condition $|u_\infty|\leq\delta$
implies both of them.
Then, by virtue of \eqref{small-22} together with \eqref{small-33},
all the assumptions in Proposition \ref{identify} are fulfilled.
We thus obtain the decay property
\[
\|w(t)\|_\infty=O(t^{-1/2}) \qquad
\mbox{as $t\to\infty$}
\]
which together with \eqref{auxi-decay}
leads us to \eqref{attain-1} in view of \eqref{decompo}.
For the landing problem, it is obvious to obtain \eqref{attain-2}
without any smallness condition on the steady flow $u_s$.
We have thus completed the proof of both Theorems
\ref{starting} and \ref{landing}.
\hfill
$\Box$
\bigskip

\noindent
{\bf Acknowledgments}.
T.H. is supported by Grant-in-Aid for Scientific Research 15K04954
``{\em Mathematical Analysis of Interaction of Motions between Viscous Incompressible Fluids and Rigid Bodies}'' from the Japan Society for the Promotion of Science.
P.M. is supported by MIUR via the PRIN (2016)
``{\em Nonlinear Hyperbolic Partial Differential Equations, Dispersive and Transport Equations: Theoretical and Applicative Aspects}''.
Most part of this work was done while T.H. stayed at
Universit\`a degli Studi della Campania Luigi Vanvitelli.
The research of both authors is partially supported by GNFM research group of the 
{\em Instituto Nazionale di Alta Matematica}.

\end{document}